\RequirePackage{fix-cm}
\documentclass[smallextended]{svjour3}       
\smartqed  

\usepackage{epsfig}
\usepackage{cuted}
\usepackage{dsfont}
\usepackage{lineno}
\usepackage{siunitx}

\usepackage{graphicx,graphics}
\usepackage{amssymb,latexsym}
\usepackage{mathrsfs}
\usepackage{setspace}
\usepackage{appendix}
\usepackage{esint}
\usepackage{ulem}
\usepackage{xcolor}

\def\RR{\mathbb{R}}

\newcommand{\eeq}{\end{equation}}
\newcommand{\beq}{\begin{equation}}

\begin{document}

\title{Evaluating the impact of increasing temperatures on changes in Soil Organic Carbon stocks: sensitivity analysis and non-standard discrete approximation 
}

\titlerunning{A model for SOC changes}        

\author{Fasma Diele \and
    Ilenia Luiso \and Carmela Marangi \and    Angela Martiradonna
}


\institute{F. Diele, I. Luiso, C. Marangi  \at
              Istituto per Applicazioni del Calcolo 'M.Picone', sede di Bari, Italy \\
              \email{fasma.diele@cnr.it, i.luiso@ba.iac.cnr.it, carmela.marangi@cnr.it}           
           \and
            A. Martiradonna \at
               Università degli Studi di Bari, Bari, Italy\\
               \email{angela.martiradonna@uniba.it}  
  }


\maketitle

\begin{abstract}
A novel model is here introduced for the {\it SOC change index} defined   as the normalized  difference between the actual Soil Organic Carbon  and the value assumed at an initial reference year. It is tailored on the RothC carbon model dynamics and assumes as baseline the value of the  SOC equilibrium  under  constant environmental conditions.  A sensitivity analysis is performed to evaluate the response of the model to changes of  temperature, Net Primary Production (NPP), and land use soil class (forest, grassland, arable). A non-standard monthly time-stepping procedure has been proposed to approximate the SOC change index in the Alta Murgia National Park, a protected area in the Italian Apulia region, selected  as test site. In the case of arable class, the SOC change index  exhibits a negative trend which can be inverted by a suitable organic fertilization program here proposed. 
\end{abstract}

\section{Introduction}
\label{intro}
For reporting on Target 15.1, one of the seventeen Sustainable Development Goal (SDG) adopted by the United Nations \cite{minelli2017scientific} in 2015, the Good practice guidance \cite{good} indicates how to  calculate the extent of land degradation. It recommends the development and the use  of  analytical methods for measuring the three indicators which address the key aspects of land-based natural capital: trends in land cover, trends in land productivity and trends in  soil organic carbon (SOC) stocks. These indicators can assess the quantity and the quality of land-based natural capital and most of the associated ecosystem services.

\noindent Roughly speaking, SOC stock is the carbon captured by plants through photosynthesis which remains in the soil after decomposition of soil organic matter.  A decrease in SOC stocks is among the significant universal indicators for land and soil degradation  and can compromise all the efforts to achieve the SDGs especially those with reference to food, health, water, climate, and land management \cite{lorenz2019soil}. 

\noindent  Well-validated models which take into account the  interactions among climate, soil and land use management can be used to predict SOC changes under the different management and climatic conditions. The Rothamsted carbon model (RothC, \cite{coleman1996rothc}, \cite{parshotam1996rothamsted}) is one of the most commonly used tool to simulate soil organic carbon dynamics in arable, grassland and forest systems. Although it does not place the action of bacteria at the hearth of the  mechanisms of decomposition   as required by current theories \cite{lehmann2015contentious,hammoudi2015mathematical}, it is widely used because it captures the general principles of soil organic dynamics,  it is relatively simple and general, it requires relatively few parameters and can be  easily applied at scales  from  regional \cite{farina2013modification}, to global  \cite{morais2019detailed}.

\noindent In this paper, for making a scenario analysis of SOC changes, we propose a novel model tailored on RothC dynamics, which describes the 
evolution of the so-called  {\it SOC change index}. It is defined as the difference between the SOC values at the last  and the first year (as in \cite{morais2018proposal}), here normalized by  the carbon inputs generated by the total plant and the farmyard manure,  both evaluated at the initial baseline year. As test example,  we evaluate the impact of changes in temperature on the achievement  of land degradation neutrality for the  SOC indicator in the Alta Murgia National Park, a protected area in the Apulia region located in the south of Italy. It is known that the increase or decrease of the SOC stocks  under climate change will depend upon which process, in the future and in a given location,   dominates between increased plant inputs through increases in net primary production (NPP), and increased decomposition rates \cite{gottschalk2012will}. With the aim of detecting factors which determine the size and the direction of change  in the considered protected area,  
a sensitivity analysis, based on the direct method described in  \cite{dickinson1976sensitivity}, is performed.
 The sensitivity analysis 
is applied to a modified version of the {\it SOC change index model}, based on time averaged values, and 
provides
 local information on the impact of parameters change on the behavior of the system solution. In particular, we 
evaluate the impact on  the SOC change index of the variation of  three representative  parameters: mean annual temperature,  NPP annual values with respect to  reference values and degree of decomposability of plant material (the so-called {\it DPM/RPM ratio}), which in turn is related to the class of land use (forest, grassland and arable).
 
 \noindent 
Trends in SOC changes from $2005$, taken as baseline year, to $2019$, the final year,
 are  simulated by means of a monthly discrete non-standard approximation of the continuous model for forest, grassland and arable systems. It is based on the discrete non-standard monthly time stepping procedure provided in \cite{diele2021non} for solving the carbon dynamics in all of the compartments. Given the linearity of the RothC model, the SOC change can be discretized with the same matrix function of the monthly stepsize. Results obtained indicate positive trends for SOC change in case of  both forest and grassland  systems.  
When the arable class is considered without including the input of farm fertilizers, our model predicts a  negative trend of the introduced normalized SOC change variable. As  a final result, we evaluate  the optimal organic fertilization program  to invert the trend and  keep positive the SOC change.  When used with predicted climate and NPP data, the optimal fertilization program may guarantee the achievement of land degradation neutrality for the SOC indicator. 

\noindent The paper is organized as follows. In Section \ref{sec:1} we briefly describe the original RothC model and define the SOC indicator for the continuous counterpart of the original model. Moreover  we introduce a more realistic representation of the density function of the plant carbon input which can be proven to be periodic. Input data and parameters are then identified and described. In Section \ref{sec:3} we explain how the issue of determining the initial carbon input  is solved in the proposed formulation and we define a novel SOC change index which  overcome the problem. Then, in Section \ref{sec:4} we analyze the model assuming that there is no carbon input due to the organic fertilization and determine  the sensitivity of the model to the variation of the above mentioned  parameters:  temperature, NPP and land use class. The issue of a possible positive contribution of organic fertilization is faced in Section \ref{sec:5} where we propose to consider the farmyard manure input as a control variable to reach neutrality, and modify the model accordingly. To perform the simulations, we apply  a numerical non-standard technique which preserves the equilibrium state of the continuous dynamics and is  described in Section \ref{sec:6}. In Section \ref{sec:7} we present a test case illustrating the trends of SOC change in a protected area,  in the years 2005-2019, as a function of the measured changes of temperature and NPP for the three land use classes analyzed (forest, grassland, arable). Finally, in Section \ref{sec:8} we draw our conclusions.

\section{The RothC model}
\label{sec:1}
Within the RothC model, soil organic carbon   is divided into the five carbon pools noted: $c_{dpm}$, $c_{rpm}$, $c_{hum}$, $c_{bio}$ and $c_{iom}$ (see Figure \ref{fig:1}).
\begin{figure}
\begin{center}
\includegraphics[width=\textwidth]{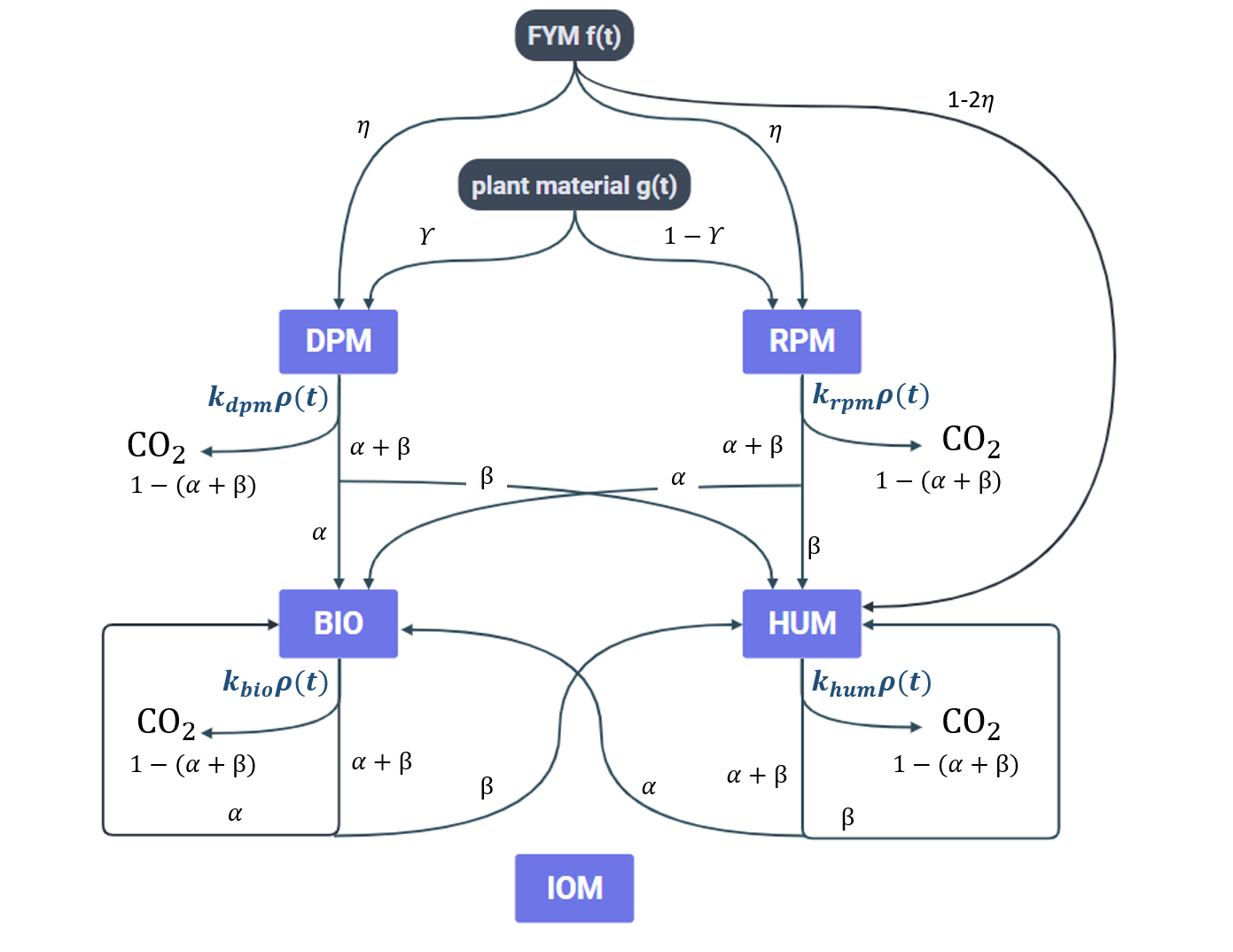} 
\caption{\small  Flow chart of the RothC model.} 
\label{fig:1}
\end{center}
\end{figure}
The already decomposed plant material is regarded as $c_{hum}$, whereas the total carbon mass of microbial organisms is represented by the $c_{bio}$ pool. All non decomposable or inert material is defined as $c_{iom}$. In general, all pools $c_{i}$ will decompose and form $CO_2$, $c_{bio}$ and $c_{hum}$.  The four active compartments  $c_{dpm}$, $c_{rpm}$, $c_{hum}$ and $c_{bio}$,  undergo decomposition as a function of different rate constants which correspond to the entries of the vector $\mathbf{k}\,=\, [k_{dpm},k_{rpm},k_{bio},k_{hum}]^\intercal$, and of the rate modifier $\rho(t)$ which depends on the clay content of the soil, on climate variables (rainfall, temperature, open pan evaporation) and land cover.  The fraction $\alpha\,+\beta$ of metabolised carbon incorporated into the sum of compartments $ c_{bio}(t)\, +\,c_{hum}(t)$ is determined by the clay content of the soil, while the remaining part $\delta:= 1-\alpha-\beta$ is 
released as $CO_2$ and lost by the system.

\noindent For the aim of what follows we denote with $T >0$ the length of a reference time interval  (generally one year) and we formulate the  RothC model as:
\begin{equation}\label{eq:RothCode}
\displaystyle \frac{d\mathbf{c}}{dt}\, = \, \rho(t)\, A\, \mathbf{c} \, +\, \mathbf{b}(t),\, \qquad t\in ]\,t_0\,+\,nT,\, t_0\,+\,(n+1)T\,]\,, \quad  n=0,\,\,\dots,
\end{equation}
where $\mathbf{c}(t)=[c_{dpm}(t),\, c_{rpm}(t),\, c_{bio}(t),\, c_{hum}(t)]^\intercal$ and 
$\mathbf{c}(t_0)=\mathbf{c}_0 \geq 0$ denotes the vector of the initial concentrations. 
The matrix $A$ is given by 
$$
A=\left(\begin{array}{cccc}
-k_{dpm} &0 & 0&0 \\\\
\displaystyle 0 & -k_{rpm}& 0 & 0\\\\
 \displaystyle \alpha\,k_{dpm}\;\; & \;\alpha\,k_{rpm}\;\;& (\alpha-1)\,k_{bio}& \alpha\,k_{hum}\\\\
\displaystyle \beta\, k_{dpm} & \beta\, k_{rpm} & \beta\, k_{bio} &(\beta-1) \,\, k_{hum}
 \end{array}\right).
$$
\noindent The vector $\mathbf{b}(t)$ represents the carbon amount entering the system at time $t$. It 
takes into account both the input of plant residues $ g(t) \, \mathbf{a}^{(g)} $ and the input of farmyard manure (FYM) $f(t)\, \mathbf{a}^{(f)}$, so that $$ \mathbf{b}(t): =\, g(t) \, \mathbf{a}^{(g)} \,+\, f(t)\, \mathbf{a}^{(f)}.$$ The entries of vectors $\mathbf{a}^{(g)}:=[\gamma, \,1\,-\, \gamma, \,0, \, 0]^\intercal$ and $\mathbf{a}^{(f)}:=[\eta, \, \eta,\, 0, \, 1\,-\,2\,\, \eta]^\intercal$ are the fraction inputs $0\, \leq \, \gamma\,\leq 1$, $0\, \leq \, \eta\,\leq 1/2$, which sum up to $1$.

\begin{definition}
\label{SOCind}
We define as SOC indicator of the continuous RothC model (\ref{eq:RothCode}) the function $SOC(t)\,= c_{iom}{(t)}\,+c_{dpm}(t)\,+ c_{rpm}(t)\,+ c_{bio}(t)\,+ c_{hum}(t)$ for $t\geq t_0$,
where $c_{iom}$ denotes the constant carbon content in the inactive compartment IOM.
\end{definition}

\noindent Although different approaches can be adopted for calculating the size of IOM, \cite{parshotammodelling},\cite{parshotam1999inert}, 
here we use the classical  equation  given by Falloon et al. in \cite{falloon1998estimating}: 
$$
c_{iom}(t)\,=\, 0.049\, SOC^{1.139}(t)
$$
so that the SOC indicator is obtained by solving the equation 
$$
 0.049\, SOC^{1.139}{(t)}\,-\, SOC{(t)}\, +\, soc(t)\, =\, 0,
$$
where
\begin{equation}\label{eq:SOC}
    soc(t):=\, c_{dpm}(t)\,+ c_{rpm}(t)\,+ c_{bio}(t)\,+ c_{hum}(t)
\end{equation}
satisfies the differential equation
 \begin{equation}\label{SOCmodel}
 \begin{array}{lcl}
  \displaystyle \frac{d soc}{dt}(t)\,     = \mathds{1}^\intercal \,\, \displaystyle \frac{d\mathbf{c}}{dt}(t) &=&\,\rho(t)\, \mathds{1}^\intercal \, A \, \mathbf{c} \, +\, g(t)\, +\, f(t)\\\\
      & =  & -\rho(t)\, \delta\, \mathbf{k}^\intercal \mathbf{c} \, +\, g(t)\, +\, f(t).
 \end{array}
 \end{equation}

\subsection{A realistic representation of $g(t)$}
\label{sec:2}
\noindent Towards a realistic analytic representation of the density function $g(t) $ of plant carbon  input, we consider that $g(t)$ can be represented as follows
\begin{equation}\label{gdiff}
    g(t)\, =\,\displaystyle P(\,t_0+n\,T)\,\,\, \hat g(t) \qquad  \forall t\in [\,t_0\,+\,nT,\, t_0\,+\,(n+1)T\,], \quad  n=0,\,\,\dots,
\end{equation}

\noindent where
\begin{equation}\label{hatg}
    \hat g(t)\,:= \displaystyle \frac{g(t)}{\displaystyle \int_{t_0\,+\,nT}^{t_0\,+\,(n+1)T} \, g(s) \, ds}.\, 
\end{equation}
The function $\hat g$ represents the density distribution of plant carbon inputs into the soil expressed as a proportion of the  total $P(\,t_0+n\,T):= \displaystyle \int_{t_0+nT}^{t_0+(n+1)T} g(s) \, ds$,  in each time interval $[t_0+nT,\, t_0+(n+1)T]$ of length $T$, for $n=0,\, 1,\, \dots$. 
In real applications  the function $\hat g(t)$ is known and, as it depends only on  seasonality, it is  well represented by an annual periodic function.
We have the following result.
\begin{theorem}\label{thm:hatg}
Set $T>0$ and suppose that $g(t)$ is a positive function which satisfies the following property
$$
g(t\,+\, T)\, =\, g(t)\,  \displaystyle \frac{\displaystyle \int_{t_0+(n+1)T}^{t_0+(n+2)T}g(s)\, ds}{\displaystyle \int_{t_0+nT}^{t_0+(n+1)T}g(s)\, ds},
$$
for all $t\in [t_0+nT,\, t_0+(n+1)T]$, and $n=0,\,1,\,\dots.$
Then, the function $\hat g(t)$, defined in (\ref{hatg}), satisfies $0\,<\, \hat g(t)\, < 1$, results  periodic with period $T$ and $\displaystyle \int_{t_0}^{t_0+T}\hat g(s)\, ds\,\,=\,\displaystyle \int_{t_0+nT}^{t_0+(n+1)T}\hat g(s)\, ds\,=1$,  for all  $n=0,\,1\,\dots$.
\end{theorem}

\begin{proof}
The result trivially follows by observing that if  $t\in  [t_0+nT,\, t_0+(n+1)T\,[$, then $t\,+\, T\, \in \,  [t_0+(n+1)T,\, t_0+(n+2)T\, [$. Consequently,
$$
\hat g(t+T)\, =\, \displaystyle \frac{g(t\,+\, T)}{\displaystyle \int_{t_0+(n+1)T}^{t_0+(n+2)T} \, g(s) \, ds}\, = \, \hat g(t), 
$$
for all $t\in [t_0+nT,\, t_0+(n+1)T\, ]$ and $ n=0,\,1,\,\dots.$
\end{proof}

\subsection{Input data and parameters} 
Let us identify all the input data necessary to the RothC dynamics. 
\begin{itemize}
    \item Input per unit time ($month$) of plant residues $g(t)\,  [ t\, C\, ha^{-1}\, month^{-1}]$ and farmyard manure $f(t) [t\, C \,ha^{-1}\, month^{-1}]$, if any.

     \noindent The function $g(t)$ is supposed to be expressed as in (\ref{gdiff}). By means of Net Primary Production (NPP), it is possible to estimate 
\begin{equation}
\begin{array}{rcl}
  P(t_0+n\,T)\,   & =&  P(t_0+(n-1)\,T)\, \displaystyle \frac{NPP(t_0+n\,T)}{NPP(t_0+(n-1)\,T)}\\ \\
     & =& P(t_0)\, N_P^{(n)} \qquad  \forall  n=1,\,2\,\dots
\end{array}
    \end{equation}
     the total plant carbon input in the  year $[t_0+nT,\, t_0+(n+1)T]$, where $ N_P^{(n)}:=\displaystyle \frac{NPP(t_0+n\,T)}{NPP(t_0)}$. The function $\hat g(t)= \hat g_r(t)$ is supposed annual periodic and assuming different known shapes according to the land use. 
    \item clay content of the soil $cly$ (as a percentage);
\item  $r$ the degree of decomposability of incoming plant material, i.e. the {\it DPM over RPM ratio};
\item   air temperature $Temp(t) \, [^\circ C]$, rainfall $rain(t)\, [mm]$,  potential evapotranspiration\footnote{The original model uses open pan evaporation; here the model is used in a  modified version which makes use of potential evapotranspiration} $pet(t)$. In our tests $pet(t)$ is estimated from weather data by means of Thornthwaite's formula (see Appendix). 
\item $\mathbf{c}(t_0)\, [ \,t\, C\, ha^{-1}]$ the vector  of the initial concentrations sampled at a soil layer of depth $d\, [cm]$.
\end{itemize}

\noindent Let us identify all the parameters involved in the RothC dynamics. 
\begin{itemize}
    \item $A\,=\,A(\alpha,\beta, \mathbf{k})\, $.
From the  clay content, we can evaluate the {\it Soil Texture Factor} according to $x\,=\, 1.67\,(1.85\,+\,1.60\,e^{-0.0786\, cly})$, and consequently  $\alpha\, = \,\displaystyle \frac{0.46}{x+1}$ and $\beta\, =\, \displaystyle \frac{1}{x+1}\,-\, \alpha$;
the entries of $\mathbf{k}$\, are  given by  $k_{dpm}\,=\,10/T \,[time^{-1}],\,k_{rpm}\,=\,0.3/T\, [time^{-1}],\,k_{bio}\,=\,0.66/T\, [time^{-1}],\,k_{hum}\,=\,0.02/T\,[time^{-1}]$ .
    \item $\mathbf{b}(t)\, =\, \mathbf{b}(t,\gamma,\, \eta).$ Here $\eta\, =\, 0.49$ while $\gamma(r)\,=\, \displaystyle \frac{r}{r\,+\,1}$  varies according to the land use. Values  $0\,<\,r\,< 0.5$ of  DPM over RPM ratio are associated to the {\it forest} class, $0.5\,\leq\,r\,<1 $ to the  {\it grassland} class, $r\,\geq 1 $ to the {\it arable} class.

    \item $\rho(t)\,=\, k_a(Temp(t))\,\, k_b\left(Acc(rain(t),\,M(cly,d)\right)\,\, k_c(t,r)$.

    \noindent The modifying factor related to the temperature is generalized with respect to the original given in \cite{coleman1996rothc}, in order to assume value equal to $1$ in correspondence of the mean annual temperature  $Temp^{(0)}$  in the interval $[\,t_0,\, t_0\,+\,T\,[ $, i.e.
$$
     k_a(Temp(t))\, :=\, \displaystyle \frac{47.91}{1\,+\, e^{\displaystyle \frac{106.06}{Temp(t)\,+\,(106.06/log(46.91)\,-\,Temp^{(0)})}}},
     $$
     so that  $k_a(Temp^{(0)})\, =\, 1$.
     
     \noindent The factor $k_c(t,r)$,  associated to  the soil cover, 
   
    $$
     \quad k_c(t,r)\,=\, \left\{ \begin{array}{cl}
          0.6& \quad 0 < r <1 \\
          S_r(t)& \quad r \geq 1 ,
     \end{array}
      \right.$$
   with $S_r(t)\, =S_r(t\,+\,T)$ assuming  values between $0.6$ in the periods of the year when soil is vegetated and the maximum value $1$, when bare.

     \noindent The maximum soil moisture deficit $M$ and the point at which respiration (i.e. microorganism activity) begins
     to slow $M_b$, are defined as 
     $M:=M(cly,\, d)\,= -\, ( 20 \,+\, 1.3\,  cly\,-\, 0.01\, cly^2)\, \displaystyle \frac{d}{23}$ and  $ M_b\, =\, 0.444\, M$. The  accumulated soil moisture deficit  
     $Acc(t,M)$ is calculated 
from the first time in $[t_0+nT, \, t_0+(n+1)T]$ where evaporation $pet(t)$ exceeds rainfall 
the maximum soil moisture deficit $M$. When there is more rainfall
than evaporation, the soil will start to wet up.

\noindent  The rate modifying factor for moisture  varies between
$0.2$ and $1$ as follows
 $$
    k_b(Acc(t,M))\,: =\,  \left\{
     \begin{array}{cl}
      \, 0.2\, +\, (1-0.2) \, \displaystyle \frac{M\, -\, Acc(t,M)}{M\,-\,M_b } & \quad Acc(t,M)<M_b \\
     1 &\quad \textup{otherwise}.
     \end{array}
     \right.
     $$
     
    \end{itemize}  
     
  \section{Determining the initial plant inputs}\label{sec:3}
In all practical applications,   RothC is run in
\textquoteleft reverse mode' to calculate the initial plant inputs to the soil for the given environmental conditions. The underlying hypothesis is that the observed carbon stocks correspond to a stable constant or annual periodically varying long-term solution for their dynamics.  Once the initial plant inputs
 have been established in this way, in order to simulate future scenarios, the time changes in carbon inputs to the soil, associated  with changes in NPP (Smith et al., 2005), changes in climate conditions, or change in land use are implemented. 

\noindent Under the hypothesis that the observed carbon stocks correspond to their values at a stable equilibrium, we are going to illustrate how it is possible to avoid the first run in
\textquoteleft reverse mode' to calculate the initial plant inputs.
 Once a monitoring temporal interval $[t_0+T,\, T_f]$ is set,  by following the approach indicated in \cite{minelli2017scientific}, the baseline of $SOC$ indicator against which Land Degradation Neutrality is to be achieved, is supposed to  correspond to the carbon stocks equilibrium for averaged values of temperature, accumulate soil moisture deficit, and soil cover in a period $[t_0, \, t_0\,+\,T]$ immediately prior the monitoring time interval.    

\noindent  As concerns the average value for the   factor $k_c(t,r)$  associated to  the soil cover, it 
  can be approximated as follows:
  $$
    \quad \overline{ k_c(r)}\,=\, \left\{ \begin{array}{cl}
         0.6& \quad0 \leq r < 1 \\
        \displaystyle \displaystyle \fint_{t_{0}}^{t_{0}\,+\, T} S_r(s)\, ds \approx 0.6\,  +\, \displaystyle \frac{N_b}{30}\, &\quad r \geq 1,
    \end{array}
     \right.$$
where  $0\leq N_b\leq 12$ (generally $N_b=4$, see e.g. \cite{smith2005projected}) is the number of months per year of bare soil for arable class.
\noindent In order to have a smooth dependence on $r$, we approximate $\overline{k_c(r)}$ with the $\mathcal{C}^{\infty}$-function
\begin{equation}\label{eq:kc}
    k_c(r)\,:=\, 0.6\, +\, \displaystyle \frac{N_b}{30}\, \displaystyle \frac{e^{x(r)}}{1+e^{x(r)}},\qquad x(r)\,:=\, \displaystyle \frac{30\,(r-1)}{r}\qquad  r\, > \, 0.
\end{equation}
The function $k_c(r)$ for a generic crop related to $N_b=4$ bare months per year, is illustrated in Figure \ref{fig:rateconstant}. 
\begin{figure}[ht] 
\begin{center}
  \includegraphics[width=11cm]{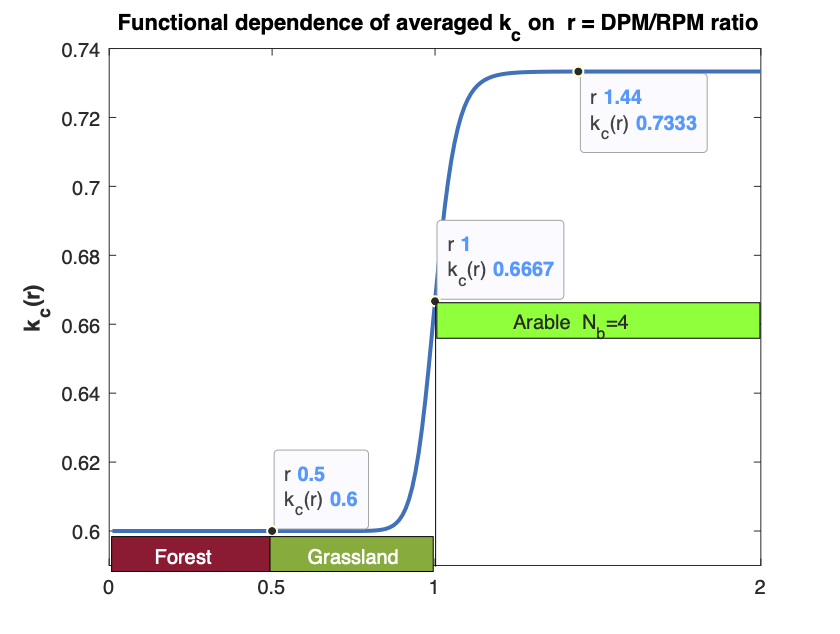}
  \caption{The rate constant modifying factor $k_c$ as a smooth function of DPM/RPM ratio. }\label{fig:rateconstant}
\end{center} 
\end{figure}

\noindent Denoting with $Temp^{(0)}$ and  ${Acc}^{(0)}$  the averaged values  for  temperature and accumulated soil deficit  on the  period $[t_0, \, t_0\,+\,T]$ assumed as reference interval, then the modifying factor $\rho(t)$ is approximated by  $\rho^{(0)}(r)\, := k_b(Acc^{(0)})\,k_c(r)$, as  $ k_a(Temp^{(0)})\, =\, 1$.

\noindent Setting  $F(t_0)\, =\displaystyle \int_{t_0}^{t_0\,+\,T}   f(s)\,ds$, then the model (\ref{eq:RothCode}), can be written as 
\begin{equation}\label{eq:first_interval}
    \displaystyle \frac{d \mathbf{c}}{dt}(t)\, =\,
\rho^{(0)}(r)\,  A \, \mathbf{c}\,+\, \displaystyle \frac{P(t_0)}{T} \, \mathbf{a}^{(g)}\, +\, \displaystyle \frac{F(t_0)}{T}\, \mathbf{a}^{(f)}, \qquad t\in]t_0,\,t_0\,+\,T].
\end{equation}

\noindent Suppose that  $\mathbf{c}(t_0)$ i.e. the distribution of the measured  $SOC(t_0)$ among compartments is known and satisfies
$$
 0.049\, SOC^{1.139}(t_0)\,-\, SOC(t_0)\, +\, \mathds{1}^\intercal \,\mathbf{c}(t_0) =\, 0.$$
We assume that $\mathbf{c}(t_0)$ is equal to the equilibrium of the dynamical system (\ref{eq:first_interval}), i.e.  
\begin{equation}\label{eq:cstarzero}
  \mathbf{c}(t_0)\, =\, \displaystyle -\frac{1}{T\,\rho^{(0)}(r)}A^{-1} \left( P(t_0)\, \mathbf{a}^{(g)}\,+\, F(t_0)\, \mathbf{a}^{(f)}\right). 
\end{equation}
Consequently,
\begin{equation}\label{eq:P0+F0}
\begin{array}{rcl}
   P(t_0)\, \mathbf{a}^{(g)}   & =& -T\,\rho^{(0)}(r)\,  A \, \mathbf{c}(t_0)\,-\, F(t_0)\, \mathbf{a}^{(f)}  \\\\
    P(t_0) \,+\, F(t_0) & =& -T\, \rho^{(0)}(r)\,\mathds{1}^\intercal \, A \, \mathbf{c}(t_0)\,=\, T\, \rho^{(0)}(r) \delta\, \mathbf{k}^\intercal\, \mathbf{c}(t_0). 
\end{array}
  \end{equation}

\noindent Under the hypothesis that $F(t_0)$ is known (i.e. the amount of the total farmyard manure used in the interval $[t_0, \, t_0\,+\, T]$), it follows that the  initial plant inputs to the soil is given by 
\begin{equation}\label{eq:P0}
\begin{array}{rcl}
     P(t_0) &=& T\,\rho^{(0)}(r)\,\delta\, \left(k_{dpm}c_{dpm}(t_0)\,+\,k_{rpm}c_{rpm}(t_0)\,\right.\\\\
     & +&\left. k_{bio}c_{bio}(t_0)\,+ \,k_{hum}c_{hum}(t_0)\right)\, -\, F(t_0)
\end{array}
\end{equation}
 Then, for all $n=1,\,2\, \dots$ the system
\begin{equation}\label{eq:model}
   \begin{array}{l}
   \displaystyle \frac{d \mathbf{c}}{dt}(t)\,=\, 
\rho(t)\,  A \, \mathbf{c}\,+\,   P(t_0+n\,T)\,\hat g_r(t) \, \mathbf{a}^{(g)}\, +\,   f(t) \, \mathbf{a}^{(f)}\\\\
P(t_0+n\,T) \,=\, P(t_0)\, N_P^{(n)}
\end{array} 
\end{equation}
 is solved for $t\in]t_0+nT,\,t_0+(n+1)T]$ starting from $\mathbf{ c}(t_0\,+\,T)\, =\,\mathbf{c}(t_0) $ given in (\ref{eq:cstarzero})  and $P(t_0)$ given in (\ref{eq:P0}),  until $t_0+(n+1)T\,\leq  T_f$. 
 
\bigskip
 \noindent For making a scenario analysis of  SOC change, which does not depend on the specific initial measured SOC value but only on the hypothesis of an initial environmental equilibrium, a useful tool is given by the {\it SOC change index} defined as the variable of change of carbon stocks normalized as follows.

\begin{definition}\label{def:delta_SOC}
We indicate with $\Delta soc_{\rho^{(0)}(r)}(t)$ the  {\it SOC change index} defined as 
$\Delta soc_{\rho^{(0)}(r)}(t):=\, \displaystyle \frac{soc(t)\, -\, soc(t_0)}{P(t_0)\,+\, F(t_0)}$
with $soc(t):= \mathds{1}^\intercal  \, \mathbf{c}(t)$, where $\mathbf{c}(t)$ solves  (\ref{eq:model}) and $P(t_0)\,+\, F(t_0)$ is given in (\ref{eq:P0+F0}).
\end{definition} 
 
 \noindent Notice that the sign of the index $\Delta soc_{\rho^{(0)}(r)}(t)$  detects if at the time $t$ the sum of  soil carbon contained in  compartments is greater than its initial value. 
 In what follows we firstly consider  the dynamics of SOC changes index when no farmayard manure input the system that is, generally the case of (not improved) grassland and forest classes.

\section{A model for SOC changes without farmyard manure input}\label{sec:4}
Soil organic carbon dynamics are driven by changes in climate and land cover or land use. In natural ecosystems, the balance of SOC is determined by gains, through plant and other organic inputs, and losses, due to the organic matter turnover \cite{smith2005projected}.
Globally, under a warming climate, increases are seen both in carbon inputs to the soil due to higher NPP, and in SOC losses due to increased decomposition. The balance between these processes defines the change in SOC stock. In some regions the processes balance, but in others, one process is affected by climate more than the other.

\noindent In order to test the effectiveness of SOC change index defined in (\ref{def:delta_SOC}) for detecting changes in SOC stock in a specific area, we deduce its temporal dynamics in Corollary \ref{thm:change_s}, preceded by the following theorem.
\begin{theorem}\label{thm:forestandgrass}
In case of no farmyard manure input, the dynamics of the variable 
$$\Delta \mathbf{c}_{\rho^{(0)}(r)}(t):= \, \displaystyle \frac{ \mathbf{c}(t)\, -\mathbf{c}(t_0)}{P(t_0)},\quad
\quad  t\in[t_0+nT,\,t_0+(n+1)T], \quad n=1, 2, \dots,$$
is governed by the equation
\begin{equation}\label{eq:change_c}
    \displaystyle \frac{d \Delta \mathbf{c}_{\rho^{(0)}(r)}}{dt}(t)\, =\,  \rho(t) \, A\, \Delta \mathbf{c}_{\rho^{(0)}(r)}\, +\, \left( \displaystyle N_P^{(n)}\, \hat g_r(t) \,-\, \displaystyle \frac{\rho(t)}{T\, \rho^{(0)}(r)}  \right) \mathbf{a}^{(g)}, 
\end{equation}
where $\Delta \mathbf{c}_{\rho^{(0)}(r)}(t_0+T)\, =\, \Delta \mathbf{c}_{\rho^{(0)}(r)}(t_0)\,=\, \mathbf{0}$ and $N_P^{(n)}=\displaystyle \frac{NPP(t_0+nT)}{NPP(t_0)}$.
\end{theorem}
\begin{proof}

\noindent In case of no farmyard manure input, by plugging the expression of $P(t_0\,+\,n T)$ into the equation for $\displaystyle \frac{d\mathbf{c}}{dt}$, the 
equation (\ref{eq:model}) becomes
$$
\displaystyle \frac{d\mathbf{c}}{dt}\,=\, 
\rho(t)\,  A \, \mathbf{c}\, +\, P(t_0)\, N_P^{(n)}\,\hat g_r(t) \, \mathbf{a}^{(g)},\qquad t\in[t_0+nT,\,t_0+(n+1)T].
$$
Thus,
$$\begin{array}{rcl} 
\displaystyle \frac{d \Delta \mathbf{c}_{\rho^{(0)}(r)}}{dt}(t) &=&\displaystyle  \frac{1}{P(t_0)} \left(\rho(t)\,  A \, \mathbf{c}\,  +\, P(t_0)\, N_P^{(n)}\,\hat g_r(t) \, \mathbf{a}^{(g)} \right) \\\\
&=&\displaystyle \rho(t) \, A \,\Delta \mathbf{c}_{\rho^{(0)}(r)}+
\frac{1}{P(t_0)} \left( \rho(t)\,  A \, \mathbf{c}(t_0)\, +\, P(t_0)\,  N_P^{(n)}\,\hat g_r(t) \, \mathbf{a}^{(g)} \right) \\\\
&=&\displaystyle \rho(t) \, A \,\Delta \mathbf{c}_{\rho^{(0)}(r)}+
 \left(   N_P^{(n)}\,\hat g_r(t) \, \mathbf{a}^{(g)} \,+\, \frac{\rho(t)\,  A \, \mathbf{c}(t_0)}{P(t_0)} \right) .
\end{array}$$
Recalling the relation between $P(t_0)$ and  $\mathbf{c}(t_0)$ in   (\ref{eq:cstarzero}) that yields   
$$
\displaystyle -\frac{P(t_0)}{T\,\rho^{(0)}(r)}\, \mathbf{a}^{(g)}\, =\,  A\, \mathbf{c}(t_0), 
$$
the result follows. 
\end{proof}

\begin{corollary}\label{thm:change_s}
In case of no farmyard manure  input, the dynamics of the SOC change index 
$\Delta soc_{\rho^{(0)}(r)}(t)$
for  $t\in]t_0+nT,\,t_0+(n+1)T]$, for $n=1, 2, \dots$,
is governed by the equation
\begin{equation}\label{eq:change_s}
   \displaystyle   \frac{d \Delta soc_{\rho^{(0)}(r)}}{dt} (t)\, =\,     -\rho(t) \, \delta \,\mathbf{k}^\intercal \Delta \mathbf{c}_{\rho^{(0)}(r)}\, +\, \left( N_P^{(n)} \, \hat g_r(t)\,-\, \displaystyle \frac{\rho(t)}{T\,\rho^{(0)}(r)}  \right)
\end{equation}
where $\Delta \mathbf{c}_{\rho^{(0)}(r)}(t)$ solves (\ref{eq:change_c}) and  $\Delta soc_{\rho^{(0)}(r)}(t_0+T)\, =\, \Delta soc_{\rho^{(0)}(r)}(t_0)\,=\, 0$.
\end{corollary}
\begin{proof}
The dynamics for $\Delta soc_{\rho^{(0)}(r)}(t)$ can be immediately deduced from the dynamics of $\Delta \mathbf{c}_{\rho^{(0)}(r)}(t)$ as  
$
\displaystyle   \frac{d \Delta soc_{\rho^{(0)}(r)}}{dt} (t)\, =\, \mathds{1}^\intercal \displaystyle \frac{d \Delta \mathbf{c}_{\rho^{(0)}(r)}}{dt}$ and observing that $\mathds{1}^\intercal A\, =\, -\delta\, \mathbf{k}^\intercal$. 
\end{proof}

\subsection{Sensitivity of the SOC change index to parameters}

\noindent In this section, we want to study  the relative importance of the different factors responsible for change in SOC stock. This will be done throughout a sensitivity analysis of SOC change index related to the dependence  on the temperature, on NPP and on the class of land use, here restricted to  forest and grassland classes. We will make use of the direct method in  \cite{dickinson1976sensitivity} where the analysis of  sensitivity is local and described by  first-order derivatives.

\noindent In this setting, $\phi \in \RR$ is a  parameter affecting the dynamics  $\displaystyle \frac{d \mathbf{y}}{dt}\, =\,\mathbf{f}(\mathbf{y}(t,\phi),\phi)$ of the $n$ dimensional variable $\mathbf{y}(t)$. The direct method requires the integration of an additional set of differential equations, together with the original system, to obtain the vector of sensitivities
 $ \mathbf{s}_{\mathbf{y},\phi}(t)$, whose components are defined as $\displaystyle \frac{\partial y_i(t,\phi)}{\partial \phi}$, i.e. 
\begin{equation}\label{eq: compact form ODE sensitivities}
\begin{array}{l}
\displaystyle \frac{d \mathbf{y}}{dt}\, =\,\mathbf{f}(\mathbf{y}(t,\phi),\phi), \qquad  \mathbf{y}(t_0,\phi)\, =\, \mathbf{y}_0(\phi)\\\\
      \displaystyle \frac{d\mathbf{s}_{\mathbf{y},\phi}}{dt}(t,\phi)= \displaystyle \frac{\partial \mathbf{f}}{\partial \phi}(\mathbf{y}(t,\phi),\phi)+\displaystyle \frac{\partial \mathbf{f}}{\partial \mathbf{y}}(\mathbf{y}(t,\phi),\phi)\;\mathbf{s}_{\mathbf{y},\phi}(t),\\\\
     \mathbf{s}_{\mathbf{y},\phi}(t_0)=\displaystyle \frac{\partial \mathbf{y}_0(\phi)}{\partial \phi},
\end{array}
\end{equation}
where $\displaystyle \frac{\partial \mathbf{f}}{\partial \mathbf{y}}$ denotes the Jacobian matrix. 

\bigskip

\noindent In order to apply the above described direct method, we need to replace the non-autonomous dynamics described in Theorem \ref{thm:forestandgrass} and Corollary \ref{thm:change_s}, with an autonomous one. Let us come back to the equation for $\Delta \mathbf{c}_{\rho^{(0)}(r)}(t)$ in (\ref{eq:change_c}). At first, we replace $Temp(t)$ and $Acc(t)$ with their averaged  values, say $Temp^{(n)}$ and $Acc^{(n)}$, in each interval $]\,t_0\,+\,nT,\, t_0\,+\,(n+1)T\,]$ so that  $\rho(t)$ can be approximated by  
$\rho^{(n)}(r)\, :=\,k_a( Temp^{(n)})\, k_b( Acc^{(n)}) \, k_c(r)$, where $k_c(r)$  given in (\ref{eq:kc}). As  $\displaystyle \fint_{t_0\,+\,nT}^{t_0\,+\,(n+1)\,T}  \hat g_r(s)\,ds\, =\, \displaystyle \frac{1}{T}$, we define the autonomous counterpart of the model (\ref{eq:change_c}) as follows:
\begin{equation}\label{eq:odedeltacn}
\begin{array}{l}
\displaystyle \frac{d\Delta \mathbf{\overline c}_{\rho^{(0)}(r)}}{dt}\, =\,
\rho^{(n)}(r)\,  A \, \Delta \mathbf{\overline c}_{\rho^{(0)}(r)}\,+\,\vartheta^{(n)}\,  \mathbf{a}^{(g)}, \\ \\ \Delta \mathbf{\overline c}_{\rho^{(0)}(r)}(t_0\,+\,T)\,=\,\mathbf{0},
\end{array}
\end{equation}
for $t\, \in \, ]\,t_0\,+\,nT,\, t_0\,+\,(n+1)T\,]$, \, $n=1,\, 2\, \dots$,
where \footnote{Let us observe that $ \vartheta^{(n)}$ does not depend on $r$, in fact $\vartheta^{(n)}\,= \,\frac{1}{T}\left(   N_P^{(n)}\, \,-\, \frac{k_a(Temp^{(n)})\;k_b(Acc^{(n)})}{k_b(Acc^{(0)})} \right)  $. \\}\\ 
\begin{equation}\label{def:vartheta}
    \vartheta^{(n)}\,:= \,\displaystyle \frac{1}{T}\left( N_P^{(n)} \, \,-\, \displaystyle \frac{\rho^{(n)}(r)}{\rho^{(0)}(r)} \right).  
\end{equation}
\noindent With the previous notations, we define
\begin{definition}\label{sensitivity_SOC}
The {\it sensitivity of the SOC change index to the parameter $\phi$} is defined as the sum of the entries of the vector 
 $\mathbf{s}_{\Delta \mathbf{\overline c},\phi}$, which is the sensitivity to the parameter $\phi$  of the variable $\Delta \mathbf{\overline c}_{\rho^{(0)}(r)}(t)$, whose dynamics is described in (\ref{eq:odedeltacn}).
\end{definition}


\noindent In the following we are going to analyze the sensitivity  of SOC change index to three different parameters: $Temp^{(1)}$ representing the annual averaged temperature, $N_P^{(1)}:=NPP(t_0+T)/NPP(t_0)$ representing the NPP input normalized by the value at the reference year, and $r$ related to change of land use, from forest (lowest values of $r$) to arable (highest value of $r$).

 \noindent Before proceeding we provide the following result useful for the sensitivity analysis of the SOC change index to parameters $Temp^{(1)}$ and $r$ in the time interval $]\,t_0\,+\,T,\, t_0\,+\,2\,T\,]$. 
\begin{theorem}\label{thm:delta c function n=1}
The solution of the initial value problem (\ref{eq:odedeltacn}) in the time interval  $]t_0\,+\,T,\, t_0\,+\,2\,T\,]$ is given by 
\begin{equation}\label{eq:delta c function n=1}
  \Delta \mathbf{\overline c}_{\rho^{(0)}(r)}(t)\,=\,(t-t_0-T)\,\vartheta^{(1)}\,\varphi\left( (t-t_0-T)\,\rho^{(1)}(r)\, A\,\right)\mathbf{a}^{(g)},
\end{equation}
where $\varphi(z):=z^{-1}(e^z-1)$.
\end{theorem}
\begin{proof}
Since in each interval  equation (\ref{eq:odedeltacn}) 
corresponds to  an autonomous, non homogeneous and linear differential system, the initial value problem in correspondence of $n=1$, has a unique solution given by
$$\begin{array}{lcl}
\Delta \mathbf{\overline c}_{\rho^{(0)}(r)}(t) &=&e^{\rho^{(1)}(r) A(t-(t_0+T))}\Delta \mathbf{\overline c}_{\rho^{(0)}(r)}(t_0+T)\\\\
&+&e^{\rho^{(1)}(r) A(t-(t_0+T))}\displaystyle\int_{t_0+T}^{t} e^{-\rho^{(1)}(r) A\tau}\,\vartheta^{(1)}\, \mathbf{a}^{(g)} d\tau \\\\
&=&\vartheta^{(1)}\, \,e^{\rho^{(1)}(r) A(t-t_0-T)}\,\displaystyle\left(\int_{t_0+T}^{t}\!\!\!e^{-\rho^{(1)}(r) A\tau}d\tau\right)\mathbf{a}^{(g)}  \\\\
&=&\vartheta^{(1)} \,e^{\rho^{(1)}(r) A(t-t_0-T)}\displaystyle \frac{A^{-1}}{\rho^{(1)}(r)}\left(I-e^{-\rho^{(1)}(r) A(t-t_0-T)}\right)\mathbf{a}^{(g)}.
\end{array}
$$
 By observing that the matrices $e^{\rho^{(1)}(r) A(t-t_0-T)}$ and $A^{-1}$ commute, we have that
$$
      \Delta \mathbf{\overline c}_{\rho^{(0)}(r)}(t)=\vartheta^{(1)}\,\displaystyle \frac{A^{-1}}{\rho^{(1)}(r)}\left(e^{\rho^{(1)}(r) A(t-t_0-T)}-I \right)\mathbf{a}^{(g)} .
$$

\end{proof}

\subsection{Sensitivity of the SOC change index to the parameter $Temp^{(1)}$ }\label{subsection:sensitivity to Temp}
Accordingly to Definition \ref{sensitivity_SOC}, we define   the sensitivity of the SOC change index  to $Temp^{(1)}$
 the quantity $s_{\Delta soc,Temp^{(1)}}:=  \mathds{1}^\intercal \mathbf{s}_{\Delta \mathbf{\overline c},Temp^{(1)}}$. The following theorem holds.
 \begin{theorem}\label{thm: sensitivity to Temp}
 The sensitivity of the SOC change index   to $Temp^{(1)}$ satisfies the following differential equation
 \begin{equation}\label{eq:ode delta SOC sensitivity to Temp^n}
\begin{array}{ll}
    \displaystyle \frac{d s_{\Delta soc,Temp^{(1)}}}{dt} &=  -\rho^{(1)}(r)\delta\, \mathbf{k}^\intercal \, \mathbf{s}_{\Delta \mathbf{\overline c},Temp^{(1)}}\\\\
&-\displaystyle \frac{\partial\rho^{(1)}(r)}{\partial Temp^{(1)}}\left(\delta \,\mathbf{k}^\intercal \Delta \mathbf{\overline c}_{\rho^{(0)}(r)}+\displaystyle \frac{1}{T\rho^{(0)}(r)} \right)
\end{array}
\end{equation}
for $t\,\in\,]\,t_0+\,T \,,\,t_0+2\,T\,]$, with the initial condition
$$s_{\Delta soc,Temp^{(1)}}(t_0+T) = 0.$$
Moreover, there exists an $\epsilon>0$ such that for all $t\,\in \,[ \,t_0+T,t_0+T+\epsilon\,]$
$$s_{\Delta soc,Temp^{(1)}}(t)\leq 0.$$ 
 \end{theorem}
 \begin{proof}
Since the sensitivity of $\Delta soc_{\rho^{(0)}(r)}$ to $Temp^{(1)}$ is defined as $s_{\Delta soc,Temp^{(1)}}= \mathds{1}^\intercal \mathbf{s}_{\Delta \mathbf{\overline c},Temp^{(1)}}  $, let us begin by obtaining the initial value problem for $\mathbf{s}_{\Delta \mathbf{\overline c},Temp^{(1)}}$. According to equations (\ref{eq: compact form ODE sensitivities}), applied to equations (\ref{eq:odedeltacn}) for $t\,\in\,]\,t_0+\,T \,,\,t_0+2\,T\,]$ (i.e. $n=1$), we have that
\begin{equation}\label{eq:ivp sens delta c to temp}
\begin{array}{l}
\begin{array}{ccl}
\displaystyle \frac{d\mathbf{s}_{\Delta \mathbf{\overline c},Temp^{(1)}} }{dt}&=&\rho^{(1)}(r)\, A\, \mathbf{s}_{\Delta \mathbf{\overline c},Temp^{(1)}}\\
&+&
\displaystyle \frac{\partial}{\partial Temp^{(1)}}\left(\rho^{(1)}(r)\,  A \, \Delta \mathbf{\overline c}_{\rho^{(0)}(r)}\,+\,\vartheta^{(1)}\,  \mathbf{a}^{(g)} \right), 
\end{array}\\\\
\mathbf{s}_{\Delta \mathbf{\overline c},Temp^{(1)}}(t_0+T)=\displaystyle \frac{\partial \Delta\mathbf{\overline c}_{\rho^{(0)}(r)}(t_0+T)}{\partial Temp^{(1)}}=\mathbf{0},
\end{array}
\end{equation}
where\\

$\displaystyle \frac{\partial}{\partial Temp^{(1)}} \left(\rho^{(1)}(r)\,  A \, \Delta \mathbf{\overline c}_{\rho^{(0)}(r)}\,+\,\vartheta^{(1)}\,  \mathbf{a}^{(g)} \right)\;=\;$
$$\begin{array}{ll}
\qquad\qquad\;&=\;  \displaystyle \frac{\partial \rho^{(1)}(r) }{\partial Temp^{(1)}} A \, \Delta \mathbf{\overline c}_{\rho^{(0)}(r)}\,+\, \displaystyle \frac{\partial \vartheta^{(1)}}{\partial Temp^{(1)}} \mathbf{a}^{(g)}\\\\
&=\; \displaystyle \frac{\partial \rho^{(1)}(r) }{\partial Temp^{(1)}}\left(  A \, \Delta \mathbf{\overline c}_{\rho^{(0)}(r)}\,-\displaystyle \frac{\mathbf{a}^{(g)}}{T\rho^{(0)}(r)} \right).
\end{array}$$
Thus, for all $t\,\in\,]\,t_0+\,T \,,\,t_0+2\,T\,]$
$$
\begin{array}{l}
\displaystyle \frac{d\mathbf{s}_{\Delta \mathbf{\overline c},Temp^{(1)}} }{dt}= \rho^{(1)}(r) A \mathbf{s}_{\Delta \mathbf{\overline c},Temp^{(1)}}+
\displaystyle \frac{\partial \rho^{(1)}(r) }{\partial Temp^{(1)}}\left(  A \, \Delta \mathbf{\overline c}_{\rho^{(0)}(r)}\,-\displaystyle \frac{\mathbf{a}^{(g)}}{T\rho^{(0)}(r)} \right).
\end{array}
$$
By multiplying both sides of the previous equation by $\mathds{1}^\intercal,$ and by recalling that $\mathds{1}^\intercal A = -\delta \,\mathbf{k}^\intercal,$ and $ \mathds{1}^\intercal \mathbf{a}^{(g)}=1,$ equation (\ref{eq:ode delta SOC sensitivity to Temp^n}) is proved.

\bigskip 
\noindent For proving the second part of the statement, let us consider the expression of $\Delta \mathbf{\overline c}_{\rho^{(0)}(r)}(t) $ obtained in  Theorem \ref{thm:delta c function n=1}. \\By setting
$\psi(t):=\mathbf{k}^\intercal\,\varphi\left( \rho^{(1)}(r) A\,(t-t_0-T)\right)\mathbf{a}^{(g)}$, we have that 
$$
\mathbf{k}^\intercal \Delta \mathbf{\overline c}_{\rho^{(0)}(r)}(t) = 
\,(t-t_0-T)\, \psi(t)\, \vartheta^{(1)},
$$
\noindent and, by replacing $\vartheta^{(1)}$ with Definition \ref{def:vartheta}, equation (\ref{eq:ode delta SOC sensitivity to Temp^n}) becomes
$$\begin{array}{ll}
\displaystyle \frac{d s_{\Delta soc,Temp^{(1)}}}{dt} &= -\rho^{(1)}(r)\, \delta \,\mathbf{k}^\intercal \mathbf{s}_{\Delta \mathbf{\overline c},Temp^{(1)}}\\\\
\qquad&-
 \displaystyle \frac{\partial\rho^{(1)}(r)}{\partial Temp^{(1)}}\!\left[
 \displaystyle \frac{ \delta (t-t_0-T) \psi(t)}{T}\left( N_P^{(1)}- \displaystyle \frac{\rho^{(1)}(r)}{\rho^{(0)}(r)}\right)\!+  \displaystyle \frac{1}{T\rho^{(0)}(r)} \right].
\end{array}
$$
Consider that $\psi(t_0+T)= \mathbf{k}^\intercal \mathbf{a}^{(g)} \, >0$, then, by continuity,
 there exists an $\epsilon>0$ such that 
$\psi(t)>0$ for all $t\in]t_0+T, t_0+T+\epsilon]$.
By defining $k_{min}:=\min_i{\mathbf{k}_i}$, then 
 $\mathbf{k}^\intercal \mathbf{s}_{\Delta \mathbf{\overline c},Temp^{(1)}}\leq k_{min}\mathds{1}^\intercal \mathbf{s}_{\Delta \mathbf{\overline c},Temp^{(1)}}$, and $\delta (t-t_0-T)\, \psi(t)\,N_P^{(1)}>0$ for all $t\in]\,t_0+T,t_0+T+\bar \epsilon\,]$ It follows that
$$\begin{array}{rcl}
\displaystyle \frac{d s_{\Delta soc,Temp^{(1)}}}{dt} &\leq & -\rho^{(1)}(r)\, \delta \,k_{min}\mathds{1}^\intercal \mathbf{s}_{\Delta \mathbf{\overline c},Temp^{(1)}}\\\\
&-&\,
 \displaystyle \frac{\partial\rho^{(1)}(r)}{\partial Temp^{(1)}} \displaystyle \frac{1- \delta\, (t-t_0-T) \, \psi(t)\, \rho^{(1)}(r) }{T\rho^{(0)}(r)}
\end{array}
$$
for all $t\in]\,t_0+T,t_0+T+\epsilon\,]$.
\noindent By continuity,  the function $(t-t_0-T) \, \psi(t) $ is positive for all $t\in]\,t_0+T,t_0+T+ \bar \epsilon\,]$ and it is equal to zero at $t=t_0+T$. Since $\displaystyle \frac{1}{\delta \, \rho^{(1)}(r)}>0$, there exists an $\epsilon>0$ such that
$(t-t_0-T) \psi(t) \leq \displaystyle \frac{1}{\delta \, \rho^{(1)}(r)}$ for all $t\in]\,t_0+T,t_0+T+ \epsilon\,]$. Thus, exploiting the positivity\footnote{  $\frac{\partial\rho^{(1)}(r)} {\partial Temp^{(1)}} \,=\,\frac{106.06}{47.91} (k_a(Temp^{(1)}))^2 k_b(Acc^{(1)})\,k_c(r) \frac{e^{\frac{106.06}{Temp^{(1)}+\frac{106.06}{log(46.91)} -Temp^{(0)} } }}{(Temp^{(1)}+\frac{106.06}{log(46.91)} -Temp^{(0)})^2}>0.$} of $\displaystyle \frac{\partial\rho^{(1)}(r)}{\partial Temp^{(1)}}$, we have that 
$$\begin{array}{l}
\displaystyle \frac{d s_{\Delta soc,Temp^{(1)}}}{dt} \leq -\rho^{(1)}(r)\, \delta \,k_{min}\,\mathds{1}^\intercal \mathbf{s}_{\Delta \mathbf{\overline c},Temp^{(1)}},\quad \forall t\in]\,t_0+T,t_0+T+ \epsilon\,]\\\\
s_{\Delta soc,Temp^{(1)}}(t_0+T)=0.
\end{array}
$$
\noindent The solution of the Cauchy problem
$\displaystyle \frac{dx}{dt}=-\rho^{(1)}(r)\,\delta\,k_{min}\, x$, with $ x(t_0+T)=0$, is the function $x(t)\equiv 0$, for all $t\,\in \,[\,t_0+T,t_0+T+\epsilon\,]$.
Since $s_{\Delta soc,Temp^{(1)}}(t_0+T)\leq x(t_0+T)\, =0 $, we have that $s_{\Delta soc,Temp^{(1)}}\leq x(t)=0,$ for all $t\,\in \,[\,t_0+T,t_0+T+\epsilon\,]$.\\
 \end{proof}

\begin{remark}\label{remark: temp1}
For sufficiently small values of $t$, the sensitivity  of SOC change index to $Temp^{(1)}$ is a negative function of time. Consequently, an initial increase in annual averaged temperature  $Temp^{(1)}$ decreases the null initial  value of $\Delta soc_{\rho^{(0)}(r)}$. Recalling that the sign of the index $\Delta soc_{\rho^{(0)}(r)}(t)$ detects if at the time $t$ the sum of  soil carbon contained in  compartments is greater than its initial value, we conclude that an initial increase in annual averaged temperature  $Temp^{(1)}$
has a negative effect on the achievement of land degradation neutrality. 
\end{remark}
\subsection{Sensitivity of the SOC change index to the $N_P^{(1)}$ ratio}\label{subsection:sensitivity to N}
According to Definition \ref{sensitivity_SOC}, the sensitivity  of SOC change index  to $N_P^{(1)}$
is given by $s_{\Delta soc,N_P^{(1)}}:=  \mathds{1}^\intercal \mathbf{s}_{\Delta \mathbf{\overline c},N_P^{(1)}}$. The following theorem holds.
\begin{theorem}\label{thm: sensitivity to N}
 The sensitivity of the SOC change index  to  $N_P^{(1)}$ satisfies the following initial value problem
 \begin{equation}\label{eq:ode delta SOC sensitivity to N}
\begin{array}{l}\displaystyle \frac{d s_{\Delta soc,N_P^{(1)}}}{dt}  = -\rho^{(1)}(r)\, \delta\, \mathbf{k}^\intercal\, \mathbf{s}_{\Delta \mathbf{\overline c},N_P^{(1)}}+\displaystyle \frac{1}{T},\quad t\,\in\,]\,t_0+T ,\,t_0+2\,T\,]
\\\\
 s_{\Delta soc,N_P^{(1)}}(t_0+T)=0.
\end{array}
\end{equation}
Moreover, $s_{\Delta soc}(t)\geq 0$ for all $t\,\in \,[\,t_0+T,t_0+2T]$. 
 \end{theorem}
\begin{proof}
At first, let us consider the sensitivity of $\Delta \mathbf{\overline c}_{\rho^{(0)}(r)}$ to $N_P^{(1)}$, which satisfies the following initial value problem
\begin{equation}\label{eq:ivp sens delta c to N}
\begin{array}{l}
     \displaystyle \frac{d \mathbf{s}_{\Delta \mathbf{\overline c},N_P^{(1)}}}{dt}=\rho^{(1)}(r)\,A\, \mathbf{s}_{\Delta \mathbf{\overline c},N_P^{(1)}}+\displaystyle \frac{\mathbf{a}^{(g)}}{T},\quad t\,\in \,]\,t_0+T,\,t_0+2\,T\,] \\\\
     \mathbf{s}_{\Delta \mathbf{\overline c},N_P^{(1)}}(t_0+T)=\mathbf{0},
\end{array}
\end{equation}
according to equations (\ref{eq: compact form ODE sensitivities}) applied to equations (\ref{eq:odedeltacn}). \\
By recalling that
 $\mathds{1}^\intercal A =-\delta\, \mathbf{k}^\intercal $ and $\mathds{1}^\intercal \mathbf{a}^{(g)}=1$ it is easy to see that $s_{\Delta soc,N_P^{(1)}}$ satisfies the initial value problem (\ref{eq:ode delta SOC sensitivity to N}).\\

\noindent  To complete the proof, let us define $k_{max}:=\max_i{\mathbf{k}_i}$. Thus,
$$ \frac{d s_{\Delta soc,N_P^{(1)}}}{dt}\geq - \,\rho^{(1)}(r)\,\delta \,\,k_{max} \,s_{\Delta soc,N_P^{(1)}}, 
$$
for all $t\in]\,t_0+T,\,t_0+2\,T\,].$ Since 
$s_{\Delta soc,N_P^{(1)}}(t_0\,+\,T)= 0$, we have that 
$s_{\Delta soc,N_P^{(1)}}\geq 0$ for all $t\,\in\,[\,t_0+T,\,t_0+2\,T\,]$.\\
\end{proof}

\begin{remark}
The sensitivity of the SOC change index to $N_P^{(1)}$ is positive, consequently an increase of the  $N_P^{(1)}$ ratio increases the null initial value of $\Delta soc_{\rho^{(0)}(r)}$. Recalling that the sign of the index $\Delta soc_{\rho^{(0)}(r)}(t)$ detects if at the time $t$ the sum of  soil carbon contained in  compartments is greater than its initial value, we conclude that an increase in annual NPP values  
has a positive effect on the achievement of land degradation neutrality. 
\end{remark}

\subsection{Sensitivity of the SOC change index to the parameter $r$ }\label{subsection:sensitivity to r}
According to Definition \ref{sensitivity_SOC}, the sensitivity of the SOC change index  to $r$
is given by $s_{\Delta soc,r}:=  \mathds{1}^\intercal \mathbf{s}_{\Delta \mathbf{\overline c},r}$. The following theorem holds.

\begin{theorem}\label{thm: sensitivity to r}
 The sensitivity of the SOC change index  to   $r$ satisfies the following initial value problem
 \begin{equation}\label{eq:ode delta SOC sensitivity to r}
\begin{array}{l}\displaystyle \frac{d s_{\Delta soc,r}}{dt}  = -\rho^{(n)}(r) \,\delta\, \mathbf{k}^\intercal \mathbf{s}_{\Delta \mathbf{\overline c},r}-
\displaystyle \frac{\partial\rho^{(n)}(r)}{\partial r} \,\delta\, \mathbf{k}^\intercal \Delta \mathbf{\overline c}_{\rho^{(0)}(r)}\\\\
 s_{\Delta soc,r}(t_0+T)=0,
\end{array}
\end{equation}
for $t\,\in\,]\,t_0+n\,T \,,\,t_0+(n+1)\,T\,],\, n=1,2,\dots$. \\
Moreover, if $\vartheta^{(1)}$ is positive, then there exists an $\epsilon>0$ such that $s_{\Delta soc,r}(t)\leq 0$ for all $t\,\in \,[ \,t_0+T,t_0+T+\epsilon\,]$. Conversely, if $\vartheta^{(1)}$ is negative, then there exists an $\epsilon>0$ such that $s_{\Delta soc,r}(t)\geq 0$ for all $t\,\in \,[ \,t_0+T,t_0+T+\epsilon\,].$ 
 \end{theorem}
 \begin{proof}
 Let us begin by obtaining the initial value problem for $\mathbf{s}_{\Delta \mathbf{\overline c},r}$. According to equations (\ref{eq: compact form ODE sensitivities})  applied to equations (\ref{eq:odedeltacn}), we have that
\begin{equation}\label{eq:ivp sens delta c to r}
\begin{array}{l}
\displaystyle \frac{d\mathbf{s}_{\Delta \mathbf{\overline c},r}}{dt}= \rho^{(n)}(r) \, A\, \mathbf{s}_{\Delta \mathbf{\overline c},r}
+
\displaystyle \frac{\partial}{\partial r}\left(\rho^{(n)}(r) \,  A \, \Delta \mathbf{\overline c}_{\rho^{(0)}(r)}\,+\,\vartheta^{(n)}\,  \mathbf{a}^{(g)} \right)
\\\\
\mathbf{s}_{\Delta \mathbf{\overline c},r} (t_0+T)=\displaystyle \frac{\partial \Delta\mathbf{\overline c}_{\rho^{(0)}(r)}(t_0+T)}{\partial r}=  \mathbf{0},
\end{array}
\end{equation}
where
$\displaystyle \frac{\partial}{\partial r}\left(\rho^{(n)}(r)\,  A \, \Delta \mathbf{\overline c}_{\rho^{(0)}(r)}\,+\,\vartheta^{(n)}\,  \mathbf{a}^{(g)} \right)=$
$$\begin{array}{ll}
\qquad\qquad\quad&=\;  \displaystyle \frac{\partial \rho^{(n)}(r)  }{\partial r}  A \, \Delta \mathbf{\overline c}_{\rho^{(0)}(r)}\,+\, \vartheta^{(n)}\displaystyle \frac{\partial \, \mathbf{a}^{(g)}}{\partial r} \\\\
&=\; \displaystyle \frac{\partial \rho^{(n)}(r)  }{\partial r}  A \, \Delta \mathbf{\overline c}_{\rho^{(0)}(r)}\,+\, \displaystyle \frac{\vartheta^{(n)}}{(r+1)^2}\mathbf{v},
\end{array}$$
and $\mathbf{v}:= [1,\, -1,\, 0,\, 0]^\intercal$. Thus, we have that
$$
\displaystyle \frac{d\mathbf{s}_{\Delta \mathbf{\overline c},r}}{dt}= \rho^{(n)}(r) \,A\, \mathbf{s}_{\Delta \mathbf{\overline c},r} + \displaystyle \frac{\partial\rho^{(n)}(r)}{\partial r}\, A\, \Delta \mathbf{c}_{\rho^{(0)}(r)} +  \displaystyle \frac{\vartheta^{(n)}\mathbf{v}}{(r+1)^2}.$$
By multiplying both sides of the above equation by $\mathds{1}^\intercal,$ and recalling that $\mathds{1}^\intercal A = -\delta \, \mathbf{k}^\intercal,$ and $ \mathds{1}^\intercal \mathbf{v}=0,$ equation (\ref{eq:ode delta SOC sensitivity to r}) is proved.

\bigskip 
\noindent For the second part of the proof, let us consider $n=1.$ We have that
$$\begin{array}{l}
\displaystyle \frac{d\,s_{\Delta soc,r}}{dt} = -\rho^{(1)}(r)\,\delta\,\mathbf{k}^\intercal \mathbf{s}_{\Delta \mathbf{\overline c},r}- \displaystyle \frac{\partial\rho^{(1)}(r)}{\partial r}\,\delta\, \,\mathbf{k}^\intercal \Delta \mathbf{\overline c}_{\rho^{(0)}(r)} \\\\
s_{\Delta soc,r}(t_0+T)=0,
\end{array}
$$
for all $t\in]\,t_0+T,\,t_0+2\,T\,]$. As in the proof of Theorem \ref{thm: sensitivity to Temp}, there exists an $ \epsilon>0$ such that for all $t\in]t_0+T, t_0+T+ \epsilon]$ the sign of the function
$$
\mathbf{k}^\intercal \Delta \mathbf{\overline c}_{\rho^{(0)}(r)}(t) = 
\vartheta^{(1)}\,(t-t_0-T)\,\mathbf{k}^\intercal \varphi\left( \rho^{(1)}(r) A\,(t-t_0-T)\right)\mathbf{a}^{(g)}
$$
is the same as the sign of $\vartheta^{(1)}$.
For this reason, we distinguish the two cases: $\vartheta^{(1)}\geq 0$ and $\vartheta^{(1)} < 0$.  Let us observe that $\displaystyle \frac{\partial\rho^{(1)}(r)}{\partial r}>0$\footnote{$\displaystyle \frac{\partial\rho}{\partial r}(Temp^{(n)},r)\,=\, k_a(Temp^{(n)})k_b({Acc^{(n)}})\, N_b\, \displaystyle \frac{e^{x(r)}}{r^2\left( 1+e^{x(r)}\right)^2},\quad x(r)=\displaystyle \frac{30(r-1)}{r}$} so that, when  $\vartheta^{(1)}\geq 0$, it results
$$\begin{array}{ll}
\displaystyle \frac{ds_{\Delta soc,r}}{dt} &= -\rho^{(1)}(r)\,\delta\,\mathbf{k}^\intercal \mathbf{s}_{\Delta \mathbf{\overline c},r}- \displaystyle \frac{\partial\rho^{(1)}(r)}{\partial r}\,\delta\, \,\mathbf{k}^\intercal \Delta \mathbf{\overline c}_{\rho^{(0)}(r)} \\\\
&\leq  -\rho^{(1)}(r)\,\delta\,k_{min}s_{\Delta soc,r}.
\end{array}$$
Since $s_{\Delta soc,r}(t_0\,+\,T)= 0$, we have that 
$s_{\Delta soc,r}(t) \leq 0$ for all $t\,\in \,[\,t_0+T,t_0+T+\epsilon\,]$.
If $\vartheta^{(1)}< 0$, then 
$\begin{array}{ll}
\displaystyle \frac{d s_{\Delta soc,r} }{dt} 
&\geq  -\rho^{(1)}(r)\,\delta\,k_{max}\,s_{\Delta soc,r}
\end{array}$ so that, as $s_{\Delta soc,r}(t_0+T)=0$, then $s_{\Delta soc,r}(t_0+T)\geq 0,$ for all $t\,\in \,[\,t_0+T,t_0+T+\epsilon\,]$ and this completes the proof.
 \end{proof}
\begin{remark}
For sufficiently small values of $t$, the sensitivity of the SOC change index  to $r$ has opposite  sign of 
$\vartheta^{(1)}$. This means that 
 an initial increase in the parameter $r$ increases or decreases the null initial  value of $\Delta soc_{\rho^{(0)}(r)}$ accordingly to negative or positive values of  $\vartheta^{(1)}$. More in details,   when changes in temperature increase the annual value NPP  more then the modifying factor $\rho^{(1)}(r)$, both  with respect to their initial values i.e.
$
 \frac{NPP(t_0+\,T)}{NPP(t_0)} \, \,\leq \, \frac{\rho^{(1)}(r)}{\rho^{(0)}(r)},  
$
this positively impacts all land use classes; viceversa,  when changes in temperature increase the modifying factor $\rho^{(1)}(r)$ more then the annual value NPP with respect to their initial value i.e. 
$
   \frac{NPP(t_0+\,T)}{NPP(t_0)} >  \frac{\rho^{(1)}(r)}{\rho^{(0)}(r)},
$
then SOC change  negatively impacts all the land use class. In both positive and negative case the arable land use class results the  most affected.  
\end{remark}

\section{A model for SOC changes with farmyard input as control variable}\label{sec:5}

\noindent In case of farmyard manure input,  Theorem \ref{thm:forestandgrass} is modified as follows.
\begin{theorem}\label{thm:arable}
Under the hypothesis $F(t_0)\, \neq \, 0$, the dynamics of the variable 
$\Delta \mathbf{c}_{\rho^{(0)}(r)}(t):=\,  \displaystyle \frac{ \mathbf{c}(t)\, -\mathbf{c}(t_0)}{P(t_0)\,+\, F(t_0)}$ 
for  $t\in[t_0+nT,\,t_0+(n+1)T]$, for $n=1, 2, \dots$,
is governed by the equation
\begin{equation}\label{eq:change_c_arable}
  \begin{array}{rcl}
  \displaystyle \frac{d \Delta \mathbf{c}_{\rho^{(0)}(r)}}{dt}   & =& \rho(t) \, A\, \Delta \mathbf{c}_{\rho^{(0)}(r)}\, +\, \left(  \, N_P^{(n)} \, \hat g_r(t) \,-\, \displaystyle \frac{\rho(t)}{T\, \rho^{(0)}(r)}  \right) \epsilon\,  \, \mathbf{a}^{(g)} \\\\
    &+ & \left( \displaystyle \frac{f(t)}{F(t_0)}  \,-\,  \displaystyle \frac{\rho(t) }{T\, \rho^{(0)}(r)} \right) (1-\epsilon)\, \mathbf{a}^{(f)}, \qquad \Delta \mathbf{c}_{\rho^{(0)}(r)}(t_0+T)\,=\, \mathbf{0},
\end{array} 
\end{equation}
where $0\leq\,\epsilon:= \displaystyle \frac{P(t_0)}{P(t_0)\,+\,F(t_0)}\, < \, 1$. 
\end{theorem}
\begin{proof}
\noindent By plugging the expression of $P(t_0\,+\,n T)$ into the
equation (\ref{eq:model}), for all $t\in[t_0+nT,\,t_0+(n+1)T],$ we have 
$$
 \displaystyle \frac{d  \mathbf{c}}{dt}
\,=\, 
\rho(t)\,  A \, \mathbf{c}\, +\, P(t_0)\, N_P^{(n)}\,\hat g_r(t) \, \mathbf{a}^{(g)} +\,   f(t) \, \mathbf{a}^{(f)}.
$$
Thus,
$$\displaystyle \frac{d \Delta \mathbf{c}_{\rho^{(0)}(r)}}{dt}  =\, \frac{1}{P(t_0)\,+\, F(t_0)} \left(\rho(t)\,  A \, \mathbf{c}\,  +\, P(t_0)\, N_P^{(n)}\,\hat g_r(t) \, \mathbf{a}^{(g)}\, +\,   f(t) \, \mathbf{a}^{(f)} \right) $$
$$=\rho(t) \, A \,\Delta \mathbf{c}_{\rho^{(0)}(r)}+
\frac{1}{P(t_0)+F(t_0)} \left( \rho(t)\,  A \, \mathbf{c}(t_0)\, +\, P(t_0)\,  N_P^{(n)}\,\hat g_r(t) \, \mathbf{a}^{(g)} +\,   f(t) \, \mathbf{a}^{(f)} \right) $$
$$=\rho(t) \, A \,\Delta \mathbf{c}_{\rho^{(0)}(r)}+
   \frac{P(t_0)}{P(t_0)+F(t_0)} N_P^{(n)}\,\hat g_r(t) \, \mathbf{a}^{(g)} \,+\, \frac{\rho(t)\,  A \, \mathbf{c}(t_0)}{P(t_0)\,+\, F(t_0)} \,+\, \displaystyle \frac{f(t)}{P(t_0)\,+\, F(t_0)}\, \mathbf{a}^{(f)} .
$$
Recalling the relation between $P(t_0)$ and  $\mathbf{c}(t_0)$ in   (\ref{eq:cstarzero}) that yields   
$$
A \mathbf{c}(t_0)\, =\, \displaystyle -\frac{1}{T\,\rho^{(0)}(r)} \left( P(t_0)\, \mathbf{a}^{(g)}\,+\, F(t_0)\, \mathbf{a}^{(f)}\right),  
$$
 we have
$$
\begin{array}{rcl}
 \displaystyle \frac{d \Delta \mathbf{c}_{\rho^{(0)}(r)}}{dt}   & =& \rho(t) \, A\, \Delta \mathbf{c}_{\rho^{(0)}(r)}\, +\, \left(  \, N_P^{(n)} \, \hat g_r(t) \,-\, \displaystyle \frac{ \rho(t)}{T\, \rho^{(0)}(r)}  \right) \epsilon\,  \, \mathbf{a}^{(g)} \\\\
    &+ & \left( \displaystyle \frac{f(t)}{F(t_0)}  \,-\,  \displaystyle \frac{\rho(t) }{T\, \rho^{(0)}(r)} \right) (1-\epsilon)\, \mathbf{a}^{(f)}.
\end{array}
$$
\end{proof}

\noindent The dynamics for $\Delta soc_{\rho^{(0)}(r)}(t)$ can be immediately deduced from the dynamics of $\Delta \mathbf{c}_{\rho^{(0)}(r)}(t)$ as  follows.

\begin{corollary}\label{thm:change_s_f}
In case of  farmyard manure input, the dynamics of the SOC change index 
$\Delta soc_{\rho^{(0)}(r)}(t)$
for  $t\in[t_0+nT,\,t_0+(n+1)T]$, for $n=1, 2, \dots$,
is governed by the equation
\begin{equation}\label{eq:SOCarable}
    \begin{array}{rcl}
   \displaystyle \frac{d \Delta soc_{\rho^{(0)}(r)}(t)}{dt}   &=&   -\delta\, \rho(t)\,\, \mathbf{k}^\intercal \Delta \mathbf{c}_{\rho^{(0)}(r)}\, \\\\
   &+&\, \epsilon\,\left(  \, N_P^{(n)} \, \hat g_r(t) \,-\, \displaystyle \frac{\rho(t)}{\epsilon \, T\, \rho^{(0)}(r)}  \right) + \,  (1-\epsilon)\, \displaystyle \frac{f(t)}{F(t_0)},
   \end{array}
\end{equation}
where $\Delta \mathbf{c}_{\rho^{(0)}(r)}(t)$ solves (\ref{eq:change_c_arable}) and  $\Delta soc_{\rho^{(0)}(r)}(t_0+T)\, =\, \Delta soc_{\rho^{(0)}(r)}(t_0)\,=\, 0$.
\end{corollary}
\begin{proof}
The result trivially arises recalling the definition of $\Delta soc_{\rho^{(0)}(r)}$ which gives that  
$
\displaystyle   \frac{d \Delta soc_{\rho^{(0)}(r)}}{dt} (t)\, :=\, \mathds{1}^\intercal \displaystyle \frac{d \Delta \mathbf{c}_{\rho^{(0)}(r)}}{dt}$ and by observing that $\mathds{1}^\intercal A\, =\, -\delta\, \mathbf{k}^\intercal$. 
\end{proof}
In view of Theorem \ref{thm:fdit}, we  introduce the following definition.
\begin{definition}\label{defro}
Set $$r_0(t)\, :=\, \rho(t)\, \frac{1}{1-\epsilon}\left[ \delta\, \mathbf{k}^\intercal\, \Delta \mathbf{c}_{\rho^{(0)}(r)}(t)\, +\, \frac{1}{T \rho^{(0)}(r)}\right] - \frac{\epsilon\,}{1-\epsilon}\,N_P^{(n)} \, \hat g_r(t).$$
We define the {\it modifying factor of the farmyard manure} as the quantity   $$f_0(t)\, :=\, \max \left(0,\,r_0(t)\right).$$
\end{definition}

\begin{theorem}\label{thm:fdit}
The density function of farmyard manure defined as $f(t)\,:=f_0(t)\, F(t_0)$
assures that  $soc_{\rho^{(0)}(r)}(t)\, \geq soc_{\rho^{(0)}(r)} (t_0)$ for all $t\in ]t_0+nT,t_0+(n+1)T]$ and $n=1,2,\dots.$
\end{theorem}
\begin{proof}
Notice that $\frac{1}{1-\epsilon}\,\frac{d}{dt}\Delta soc_{\rho^{(0)}(r)}(t)=-r_0(t)+\frac{f(t)}{F(t_0)}$.
Suppose $r_0(t) \geq 0$. By plugging the expression of $f(t)\, =\, f_0(t)\,  F(t_0)$ in the equation (\ref{eq:SOCarable}), we have that $\displaystyle \frac{d \Delta soc_{\rho^{(0)}(r)}(t)}{dt}   =   0$. Hence  $\Delta soc_{\rho^{(0)}(r)}(t)\, =\,0$ and consequently  $soc_{\rho^{(0)}(r)}(t)=soc_{\rho^{(0)}(r)} (t_0)$ for all $t\in ]t_0+nT,t_0+(n+1)T],\,\, n=1,2,\dots.$ When $r_0(t) < 0$ then $\displaystyle \frac{d \Delta soc_{\rho^{(0)}(r)}(t)}{dt} > 0$ and $soc_{\rho^{(0)}(r)}(t)\, >\, soc_{\rho^{(0)}(r)} (t_0)$ for all $t\in ]t_0+nT,t_0+(n+1)T],\,\, n=1,2,\dots.$
\end{proof}

\bigskip
\begin{remark}
Notice that the value $\epsilon=0$, which corresponds to $P(t_0)\, =\,0 $ (or $F(t_0)>>P(t_0)$), gives $r_0(t)\, :=\, \rho(t)\, \delta\, \mathbf{k}^\intercal\, \Delta \mathbf{c}_{\rho^{(0)}(r)}(t)\, +\,\displaystyle  \frac{1}{T \rho^{(0)}(r)}\, >0$  then  $soc_{\rho^{(0)}(r)}(t)=soc_{\rho^{(0)}(r)} (t_0)$ for all $t\in ]t_0+nT,t_0+(n+1)T],\,\, n=1,2,\dots.$ By increasing values of the parameter $\epsilon$, the value of $ \Delta soc_{\rho^{(0)}(r)}(t)$ increases. For $\epsilon\, =\, 1$ (which holds for $F(t_0)\,=0)$, equation  (\ref{eq:change_c_arable}) corresponds to the case with no farmyard manure input. Hence, by increasing $\epsilon$ from $0$ to $1$, we explore all the cases from only farmyard manure input to only plant input.   
\end{remark}
\section{A non-standard approximation of SOC changes}\label{sec:6}
In \cite{parshotam1996rothamsted} the author proved that the original discrete RothC model  in \cite{coleman1996rothc} can be thought as one-step, first-order in time,  discretization   of the continuous model (\ref{eq:RothCode}). In light of this interpretation, a novel non-standard first-order approximation which inherits  the discrete decomposition process of the original model and has the same equilibrium state of the continuous dynamics  (\ref{eq:RothCode}), was proposed in  \cite{diele2021non}. 
When applied as a  monthly time-stepping procedure,  it can be considered a suitable alternative to the original discrete RothC model. 
In  monthly units the annual length corresponds to  $T=12$ and the interval $[t_0+nT, \, t_0+(n+1)\, T]$ is  discretized  in the set of instants   $t_{m+1}^{(n)}\,:=\, t_{m}^{(n)}\,+ \Delta t_m$, with $m = 0,\dots, 11$ and $t_{0}^{(n)}\,=\,t_0\,+\, n\, T $. The step lengths are set as  $\Delta t_m\,:=\,\displaystyle \frac{T}{365} \, N_m\, \approx 1$, where $N_m$ is the number of days of the $m^{th}$ month of the $n^{th}$ year.  By denoting with $I$ the $4$ dimensional identity matrix, and setting $\mathbf{f}(\mathbf{c}; t):=  \rho(t)\,  A \, \mathbf{c}\,+\,   \mathbf{b}(t) $ 
and  $\widetilde A:= A\,(I-\Lambda)^{-1}\, =\,-(I-\Lambda)\, D\, (I-\Lambda)^{-1}$, with 
$$
 \Lambda=\left(\begin{array}{cccc}
 0 &0 & 0&0 \\
 0 &0 & 0&0 \\
 \alpha & \alpha & \alpha & \alpha\\
 \beta & \beta & \beta & \beta
 \end{array}\right), 
$$
the approximated values $\mathbf{c}_m^{(n)} \approx \mathbf{c}(t_{m}^{(n)}) $ of the solution of (\ref{eq:model}), are given by 
\begin{equation}\label{eq:first}
 \begin{array}{lll}
 \mathbf{c}^{(n)}_{m+1}& =& \mathbf{c}^{(n)}_m \,+\, \Delta t_m \,\varphi(\Delta t^{(n)}_m \,\rho({t_{m}^{(n)}})\, \widetilde A) \,\,\, \mathbf{f}(\mathbf{c}^{(n)}_m; t^{(n)}_m)  \end{array}
\end{equation}
 or, equivalently, 
\begin{equation}\label{eq:second}
 \begin{array}{lll}
 \mathbf{c}^{(n)}_{m+1} &=& F(\Delta t^{(n)}_m\,\rho(t^{(n)}_{m}))\,\,\, \mathbf{c}^{(n)}_m \,+\, \Delta t^{(n)}_m \,\varphi(\Delta t^{(n)}_m\,\rho(t^{(n)}_{m})\, \widetilde A ) \,\,\, \mathbf{b}(t^{(n)}_{m}),
 \end{array}
\end{equation}
where $F(t):= \Lambda + (I-\Lambda)\, e^{-t\, D}$ and $\Delta t^{(n)}_m \,\varphi(\Delta t^{(n)}_m\,\rho(t^{(n)}_{m})\, \widetilde A )\, =\,   \mathcal{O}(diag(\Delta t^{(n)}_m))$ \cite{diele2021non}, the function $\varphi$ being defined as in Theorem \ref{thm:delta c function n=1}. The formulation (\ref{eq:first}) emphasizes the sharing of the stationary equilibria of the continuous autonomous model $ \displaystyle \frac{d \mathbf{c}}{dt}\, =\,  \mathbf{f} (\mathbf{c})$  in case when the explicit temporal dependence is neglected and  temporal averaged quantities are exploited. Formulation (\ref{eq:second})   highlights the similarity with the discrete original RothC model which proceeds according to 
\begin{equation}\label{eq:RothCdiscreto}
 \begin{array}{lll}
 \mathbf{c}^{(n)}_{m+1} &=& F(\Delta t^{(n)}_m\,\rho(t^{(n)}_{m}))\,\,\, \mathbf{c}^{(n)}_m \,+\, \Delta t^{(n)}_m \,\, \mathbf{b}(t^{(n)}_{m}).
 \end{array}
\end{equation}

\noindent In this paper, we are interested in finding an analogous monthly time-stepping procedure for approximating the changes of $\mathbf{c}(t)$ provided by the evolution of the variable $\Delta \mathbf{c}_{\rho^{(0)}(r)}(t)$. From the observation that the homogeneous systems for  $\mathbf{c}(t)$ and $\Delta \mathbf{c}_{\rho^{(0)}(r)}(t)$ are both governed by the matrix $\rho(t) \, A$, it makes sense to use the non standard procedure described above. Consequently, the approximated values $\Delta \mathbf{c}_{m}^{(n)}  \approx \Delta \mathbf{c}_{\rho^{(0)}(r)}(t_{m}^{(n)})$ of the solution of (\ref{eq:change_c}), are given by \begin{equation}\label{eq:change_c_approx_first}
 \begin{array}{lll}
 \Delta \mathbf{c}^{(n)}_{m+1}& =& \Delta \mathbf{c}^{(n)}_{m}\,+\, \Delta t_m \,\varphi(\Delta t^{(n)}_m \,\rho({t_{m}^{(n)}})\, \widetilde A) \,\,\, \mathbf{f}( \Delta  \mathbf{c}^{(n)}_{m}; t^{(n)}_m)  \end{array}
\end{equation}
 or, equivalently, 
\begin{equation}\label{eq:change_c_approx_second}
 \begin{array}{lll}
 \Delta \mathbf{c}^{(n)}_{m+1} &=& F(\Delta t^{(n)}_m\,\rho(t^{(n)}_{m}))\,\,\, \Delta \mathbf{c}^{(n)}_m \,+\, \Delta t^{(n)}_m \,\varphi(\Delta t^{(n)}_m\,\rho(t^{(n)}_{m})\, \widetilde A ) \,\,\, \mathbf{b}(t^{(n)}_{m}),
 \end{array}
\end{equation}
where, with abuse of notation, $\mathbf{f}( \Delta \mathbf{c}_{\rho^{(0)}(r)}; t)\, =\,   \rho(t) \, A\, \Delta \mathbf{c}_{\rho^{(0)}(r)}\, +\, \mathbf{b}(t) $ and 
$$
 \mathbf{b}(t)\, = \, \left( N_P^{(n)} \,  \hat g_r(t)\,-\, \displaystyle \frac{\rho(t)}{T\,\rho^{(0)}(r)}  \right) \mathbf{a}^{(g)}
$$
in case of no farmyard manure input, while 
$$
 \mathbf{b}(t)\, = 
\left(  \, N_P^{(n)} \, \hat g_r(t) \,-\, \displaystyle \frac{\rho(t)}{T\, \rho^{(0)}(r)}  \right) \epsilon\,  \, \mathbf{a}^{(g)} \\\\
    + \left( \displaystyle \frac{f(t)}{F(t_0)}  \,-\,  \displaystyle \frac{\rho(t) }{T\, \rho^{(0)}(r)} \right) (1-\epsilon)\, \mathbf{a}^{(f)}, 
$$
where $0<\,\epsilon:= \displaystyle \frac{P(t_0)}{P(t_0)\,+\,F(t_0)}\, < \, 1$,  in the opposite case.

\noindent Finally, $\Delta soc_{\rho^{(0)}(r)}(t^{(n)}_{m})$ are approximated by  $  \Delta soc^{(n)}_{m}:= \mathds{1}^\intercal  \,  \Delta \mathbf{c}^{(n)}_{m} $, for $m = 1,\dots, 12$ and $n=1,\, 2\, \dots$.

\section{A test case: trends of $SOC$ changes in Alta Murgia National Park.}\label{sec:7}
As an application of the illustrated procedure, we analyze the change of SOC in Alta Murgia National Park, a protected area in Italian Apulia region, southern Italy, established in 2004 (see Figure \ref{fig:altamurgia}). Two parameters are fixed for all the land surface area of $68077$ ha, i.e. the depth layer is fixed at $d=23\, cm$ and the clay content is set  at the percentage $cly\, = \,50$, i.e. the value used in \cite{farina2013modification} for experiments at the experimental farm of the CRA-Cereal Research Centre (41\si{\celsius} 27' N, 15\si{\celsius} 30' E) in Foggia. 
\begin{figure}[t] 
\begin{center}
  \includegraphics[width=\textwidth]{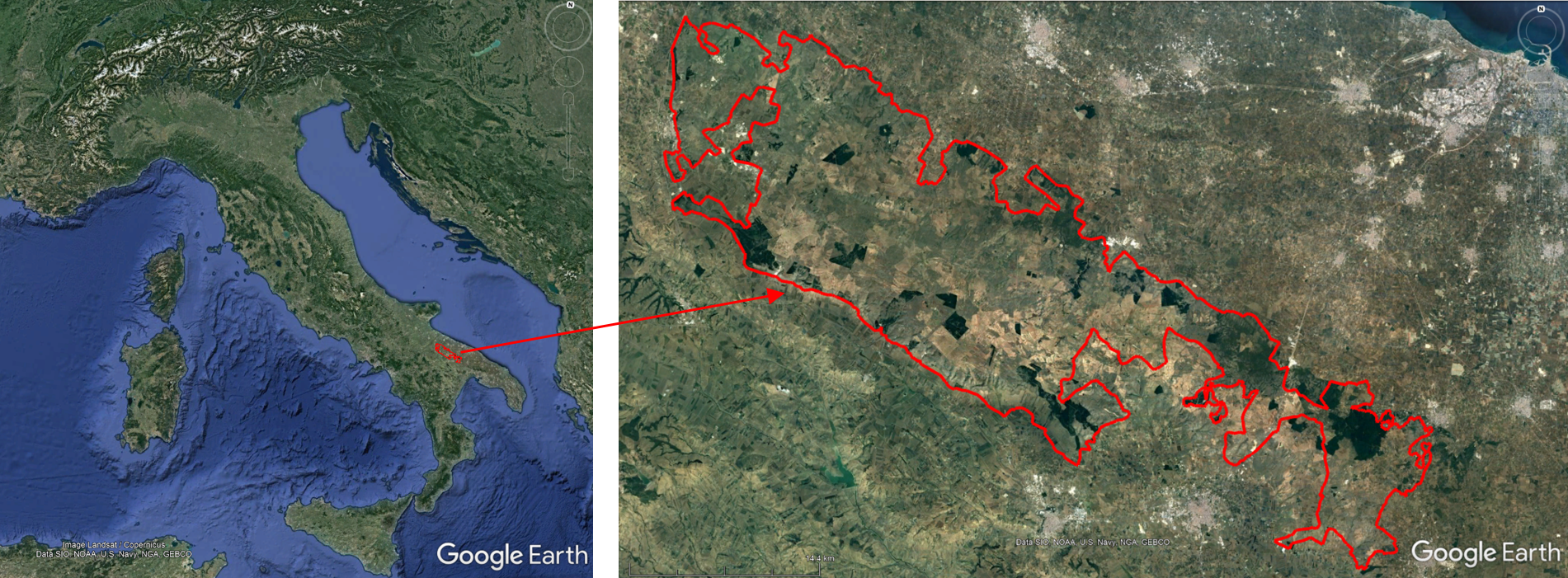}
  \caption{Boundaries of Alta Murgia National Park from Google Earth.}\label{fig:altamurgia}
\end{center} 
\end{figure}
 Temperature, rainfull, diurnal temperature range  from 2005 to 2019 at (40\si{\celsius} 75' N, 16\si{\celsius} 75' E,)   are extracted from the CRU TS 4.04 grid-box dataset \cite{harris2020version} of the  Climatic Research Unit (University of East Anglia) and NCAS (see Figure \ref{fig:weather}). Potential evapotranspiration is calculated from the available climate data according to the  Thornthwaite's formula given in the Appendix. Estimates of Net Primary Production  across Earth’s entire vegetated land surface are taken from  MOD17 project\footnote{https://www.ntsg.umt.edu/project/modis/mod17.php}, part of the NASA Earth Observation System (EOS) program, which is the first satellite-driven dataset \cite{running2019mod17a3hgf} to monitor vegetation productivity on a global scale. We have extracted NPP data in the temporal  range  from 2005 to 2019 by means of the Application for Extracting and Exploring Analysis Ready Samples (AppEEARS) \cite{2020appeears} in a polygonal containing the boundary of Alta Murgia Park (see Figure \ref{fig:npp}). 

\begin{figure}[h!] 
\begin{center} 
  \includegraphics[width=\textwidth]{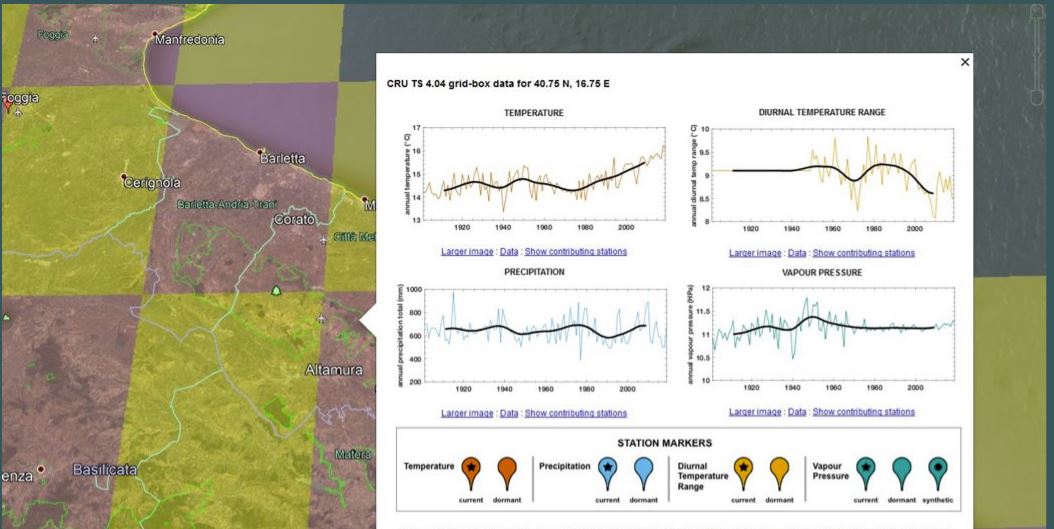}
  \caption{Climate data at (40\si{\celsius} 75' N, 16\si{\celsius} 75' E) from CRU TS 4.04 grid-box dataset of the  Climatic Research Unit (University of East Anglia).}\label{fig:weather}
  \includegraphics[width=\textwidth]{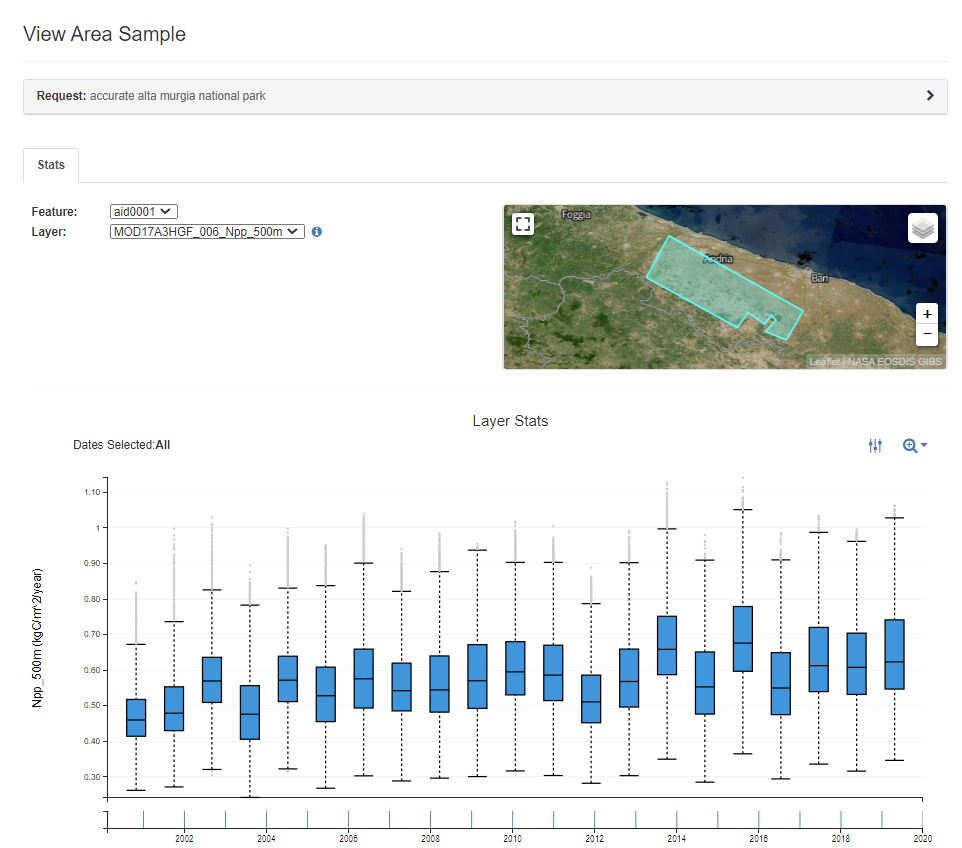}\\ 
  \caption{Selected layer and temporal values of NPP from MOD17 project of NASA EOS program. }\label{fig:npp}  
\end{center}
\end{figure}

\noindent In  Figure \ref{fig:nppvstemp} we report the annual NPP values and the averaged annual temperatures  with respect to their reference values set at $t_0\,=\,2005$, extracted by the above dataset.  As expected, to increasing temperatures correspond increasing values for NPP.

\begin{figure}[t] 
\begin{center} 
  \includegraphics[width=11cm]{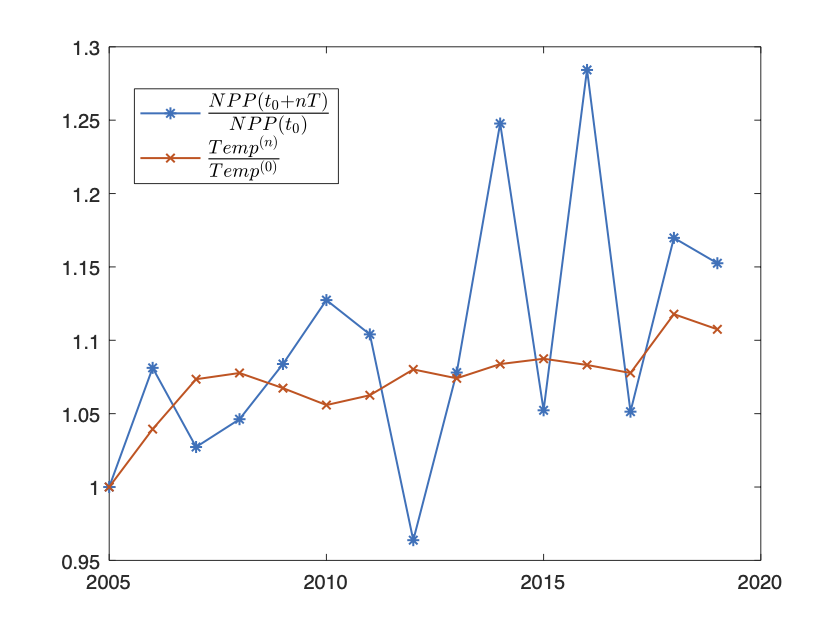}
  \caption{Behaviour of relative values of NPP and annual averaged temperatures in temporal interval $[2005,\, 2019]$ with respect to their initial values.  }\label{fig:nppvstemp}  
\end{center}
\end{figure}

\noindent Three different  formulations are used for modelling the periodic function $\hat g_r(t)$. For values of $r \in \, r(a):=\{ r\geq 1 \}$  corresponding to the arable class, we set $\hat g_r(t)= \hat g_{r(a)}(t)$; for  $r\, \in r(g):=\{0.5 \leq r < 1\}$ associated to the grassland class, $\hat g_r(t)= \hat g_{r(g)}(t)$
and we set  $\hat g_r= \hat g_{r(f)}(t)$ in correspondence of the forest class described by values $r \in r(f):=\{ 0 \leq  r \leq 0.5\}$. The monthly values at $t\, = t_{m}^{(n)}$ for $m=1,\dots\,12$,  of the three main land use  distributions $\hat g_{r(a)},\, \hat g_{r(g)}, \,\hat g_{r(f)}$ 
 are reported in Table \ref{tab:pagf}. The reported values are assumed equal to  the distribution  of plant carbon inputs  given in \cite{gottschalk2012will} which  mimics the dynamics of typical crop rotations and of permanent grassland or forest in Europe. Finally,  in Table \ref{tab:pagf} we report also the values for $k_c(t,r)$ at $t=t_{m}^{(n)}$, for the three main land use, i.e $k_c(t_{m}^{(n)},r(a)),\, k_c(t_{m}^{(n)},r(g)), \,k_c(t_{m}^{(n)},r(f)),$  assuming that the soil cover function $S_r(t)$ is periodic.  Plant cover
was assumed to occur 
in months $1$-$7$ and $12$ for the arable (croplands) class  \cite{smith2005projected}. 

 \begin{table}[t]
     \centering
  \scriptsize{
  \begin{tabular}{||l||r|r|r|r|r|r||}
\hline
\hline
t & $\hat g_{r(a)}(t)$ & $k_c(t,r(a))$ & $\hat g_{r(g)}(t)$ &$k_c(t,r(g))$ & $\hat g_{r(f)}(t)$ & $k_c(t,r(f))$\\
\hline
\hline
$t^{(n)}_1$\, (Jan, 31) & 0.0 &0.6 & 0.05&0.6 & 0.025 & 0.6\\
\hline
$t^{(n)}_2$\, (Febr, 28) & 0.0 &0.6 & 0.05 & 0.6& 0.025 & 0.6\\
\hline
$t^{(n)}_3$\, (Mar, 31) & 0.0 &0.6 &0.05 & 0.6& 0.025 & 0.6\\
\hline
$t^{(n)}_4$\, (Apr, 30) & 1/6 & 0.6 & 0.05 & 0.6 & 0.025 &0.6 \\
\hline
$t^{(n)}_5$\, (May, 31) & 1/6 & 0.6 & 0.10 & 0.6& 0.05 & 0.6\\
\hline
$t^{(n)}_6$\, (Jun, 30 ) & 1/6 & 0.6& 0.15 & 0.6& 0.05 & 0.6 \\
\hline
$t^{(n)}_7$\, (Jul, 31 ) & 0.5 & 0.6& 0.15 & 0.6& 0.05 & 0.6 \\
\hline
$t^{(n)}_8$\, (Aug, 31) & 0.0 & 1& 0.10 &0.6 & 0.05 & 0.6\\
\hline
$t^{(n)}_9$\, (Sept, 30) & 0.0 & 1& 0.10 & 0.6& 0.20 & 0.6\\
\hline
$t^{(n)}_{10}$\, (Oct, 31) & 0.0 & 1 & 0.10 & 0.6 & 0.20 & 0.6\\
\hline
$t^{(n)}_{11}$\, (Nov, 30 ) & 0.0 & 1 & 0.05 & 0.6 & 0.20 & 0.6\\
\hline
$t^{(n)}_{12}$\, (Dec, 31) & 0.0 & 0.6 &0.05 & 0.6 & 0.10 & 0.6\\
\hline
\hline
\end{tabular}}
 \caption{Monthly ($t=t^{(n)}_m$, \,  n=0,\, 1,\, 2, \dots) distribution of plant carbon inputs into the soil expressed as a proportion of the total $\hat g_r(t)$ and rate modifying factor $k_c(t,r)$ related to soil cover. Data from \cite{gottschalk2012will} and \cite{smith2005projected}.}
     \label{tab:pagf}
      \end{table}

\subsection{Numerical trends of sensitivity from 2005 to 2007}
In this section, using the Alta Murgia National Park data in the period 
2005-2007
we want to show the behaviour of the sensitivities of the SOC change index to average annual temperature, to the relative value of NPP and to $r\,=\,DPM / RPM$ ratio. We chose  $t_0\,=\,2005\, T$, with $T=12$, thus $Temp^{(1)}=14.27 $°C is the average temperature of 2006 and $N_P^{(1)}=1.08$ is the ratio between the Net Primary Production of 2006 and the Net Primary Production of 2005.
Once we have computed the numerical solution of the Cauchy problem (\ref{eq:odedeltacn})
for $n=1$, we obtain the  sensitivities by summing up the four components of the numerical solution of the initial value problems 
(\ref{eq:ivp sens delta c to temp}), (\ref{eq:ivp sens delta c to N}) and (\ref{eq:ivp sens delta c to r}), for $n=1$.\\

\noindent{The numerical approximation of the sensitivity to the average temperature in 2006, depicted in Figure \ref{fig:sensitivity to Temp_NP1_r}.a, is a negative function of time, consistently with Theorem \ref{thm: sensitivity to Temp}.
Thus, an increase in the average temperature of 2006 would have reduced $\Delta soc_{\rho^{(0)}(r)}$ during the year and, consequently, the sum of the soil carbon contained in compartments would have decreased too.
\noindent Moreover, since the sensitivity of $\Delta soc_{\rho^{(0)}(r)}$ to $Temp^{(1)}$ is a decreasing function of time, we can deduce that the perturbation in the average temperature of 2006  would have affected the rate of decomposition at every month, and this effect would have been  amplified  over time.\\}
\begin{figure}[t] 
\begin{center} 
  \includegraphics[width=12cm]{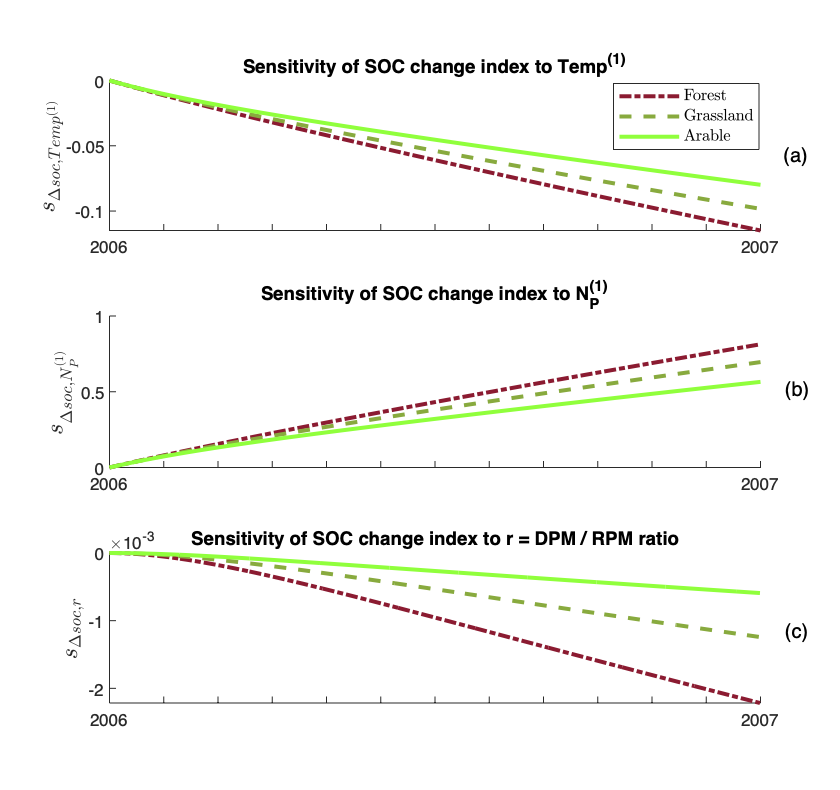}
  \end{center}
  \caption{Numerical non-standard approximation of the temporal evolution  of $s_{\Delta soc,Temp^{(1)}}$, $s_{\Delta soc,N_P^{(1)}}$ and $s_{\Delta soc,r}$ in 2006, with time-step $\Delta t=0.01$. Parameters: $r=0.25$ for the forest class, $r=0.67$ for the grassland class and $r=1.44$ for the arable class.}\label{fig:sensitivity to Temp_NP1_r}
\end{figure}

\noindent {Analogously, we can observe that the numerical approximation of the sensitivity of $\Delta soc_{\rho^{(0)}(r)}$ to $N_P^{(1)}$ is consistent with Theorem \ref{thm: sensitivity to N}. In fact in Figure \ref{fig:sensitivity to Temp_NP1_r}.b it is depicted as a positive (and increasing) function of time.
This means that an increase in the Net Primary Production in 2006 with respect to the Net Primary Production in 2005, would have increased $\Delta soc_{\rho^{(0)}(r)}$, and consequently the sum of the soil carbon contained in compartments, during the year. 
\noindent Moreover, the perturbation in $N_P^{(1)}$ would have affected the rate of decomposition at every month with this effect amplified over time although at a decreasing pace.\\}


\noindent {Finally, let us focus on the sensitivity of $\Delta soc_{\rho^{(0)}(r)}$. According to our data, $\vartheta^{(1)}=4.3620\cdot 10^{-4}$. Thus, since $\vartheta^{(1)}$ is positive, by Theorem \ref{thm: sensitivity to r} we have that the sensitivity is a negative function of time and this is consistent with Figure \ref{fig:sensitivity to Temp_NP1_r}.c.
Thus, an increase in the parameter $r$ at the beginning of 2006, i.e. a transition from forest to grassland and to arable classes, would have caused a decrease in  $\Delta soc_{\rho^{(0)}(r)}$ and the sum of the soil carbon over the  compartments during that year.
\noindent Also in this case, the perturbation in $r$ would have affected the rate of
decomposition at every month, again with an amplification of the effect  over time. }\\

{\noindent Notice that the numerical approximation of the sensitivity of $\Delta soc_{\rho^{(0)}(r)}$ to $r$ can be computed not only on the first time interval but also on the following years, by integrating 
the initial value problem (\ref{eq:ivp sens delta c to r}) together with the Cauchy problem (\ref{eq:odedeltacn}), for $t\in\,]t_0+nT,t_0+(n+1)T\,],\,n=1,\dots,14$ (see Figure \ref{fig:sensitivity to r 2006 2019}).}

\begin{figure}[t] 
\begin{center} 
  \includegraphics[width=11cm]{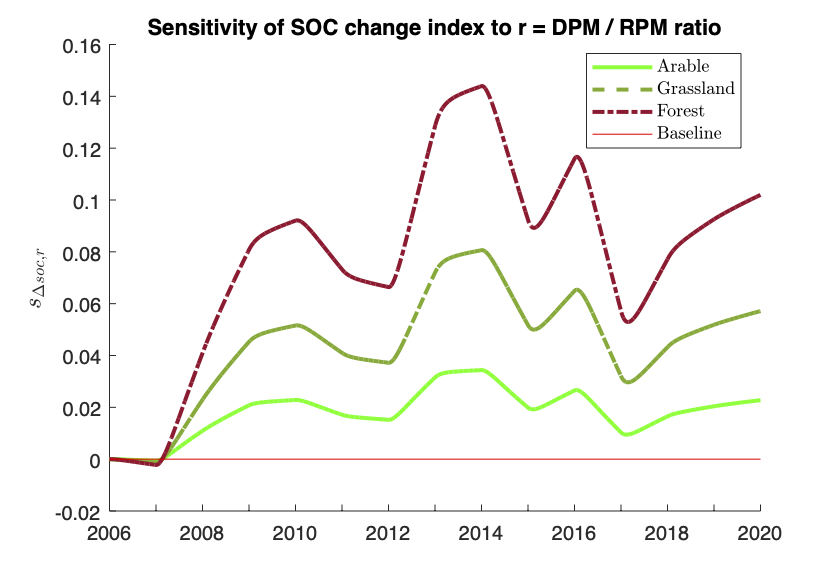}
  \end{center}
  \caption{Numerical non-standard approximation of the temporal evolution  of $s_{\Delta soc,r}$ over 14 years, with time-step $\Delta t=0.01$. Parameters: $r=0.25$ for the forest class, $r=0.67$ for the grassland class and $r=1.44$ for the arable class.}\label{fig:sensitivity to r 2006 2019}
\end{figure}

\subsection{SOC changes scenarios in years $2005$-$2019$}
We are going to illustrate the evolution of SOC changes in Alta Murgia National Park in the period $2005$-$2019$ taking as   baseline  its distribution  in $2005$ ($t_0\, =\,2005\, T$ with $T=12$). The  approximated values $\Delta \mathbf{c}_m^{(n)} \approx \Delta \mathbf{c}_{\rho^{(0)}(r)}(t_{m}^{(n)})$ of the solution of (\ref{eq:change_c}) for $t_{m}^{(n)}\, \in [t_0\,+\,n\,T, t_0\,+\, (n+1)\, T]$ with $n=1,\dots, 14$, provided by means of the non-standard discrete procedure described in (\ref{eq:change_c_approx_second}), are evaluated for the three main land use classes: forest, grassland and arable. 
For the arable case, we also show the farmyard manure program which would be able to assure the achievement of land degradation neutrality in 2019 with respect to 2005 taken as reference year. 

\subsubsection{Forest class}
For the forest class, the evolution of $\Delta soc_{\rho^{(0)}(r)}(t^{(n)}_{m})$, together with its averaged annual values, is given in Figure \ref{fig:forest}. We set $r= 1e-4$, $r=0.25$ (i.e. the value used in case of forest class in literature \cite{coleman1996rothc}), and $r=0.5$ in order to span all the values corresponding to this class. We  notice that,  for $r$ spanning the reference set $r(f)$, the trends do not differ much. However, even if it is still negative at the end of the interval, the general behaviour   of $\Delta soc_{\rho^{(0)}(r)}$ suggests that a positive value can be achieved  by $2030$. 
\begin{figure}[t] 
\begin{center} 
  \includegraphics[width=11cm]{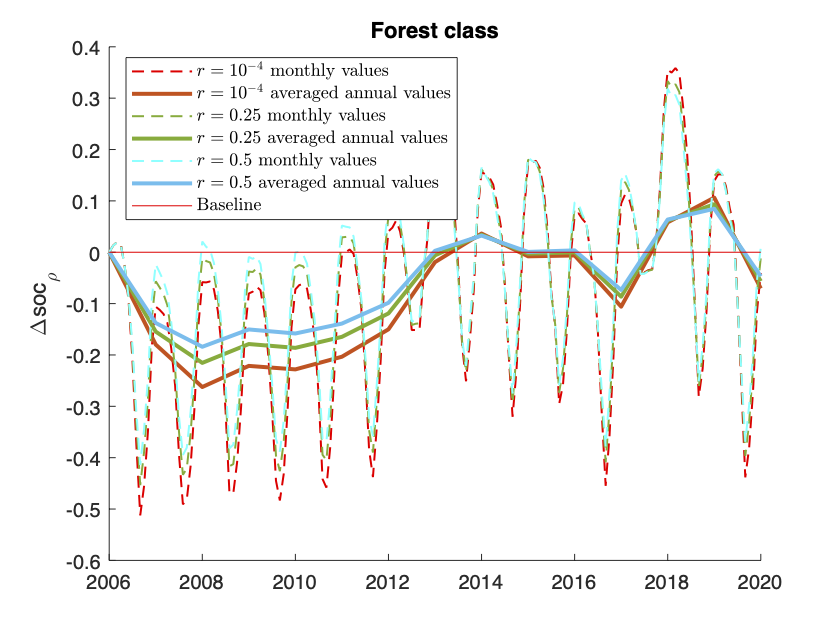}
  \end{center}
  \caption{The temporal evolution of $\Delta soc_{\rho^{(0)}(r)}(t^{(n)}_{m})$, together with its averaged annual values for forest class. Parameters $r=10^{-4}$, $r=0.25$, $r=0.5$. }\label{fig:forest}
\end{figure}
\subsubsection{Grassland class}

For the grassland  class, the evolution of $\Delta soc_{\rho^{(0)}(r)}(t^{(n)}_{m})$, together with its averaged annual values, is given in Figure \ref{fig:grassland}. We set $r=0.67$, (i.e. the value used in case of grassland class in literature \cite{coleman1996rothc}), $r=0.9$ and $r=0.95$ in order to span all the values corresponding to this class.   As for the forest class, the general trend of $\Delta soc_{\rho^{(0)}(r)}$ seems to be increasing even for a grassland scenario. Notice however that this class is much influenced by the value of $r$. For value $r=0.95$, close to the value which bounds from above the class $r(g)$, the curve reaches positive values at $2011$  and, although oscillating, it remains positive till the end of 2019. In correspondence of the value  $r=0.67$ which is the one adopted in the literature for this class, the final value is negative; however the general trend  seems to be increasing so that a positive value might be envisaged by 2030.  
\begin{figure}[t] 
\begin{center} 
  \includegraphics[width=11cm]{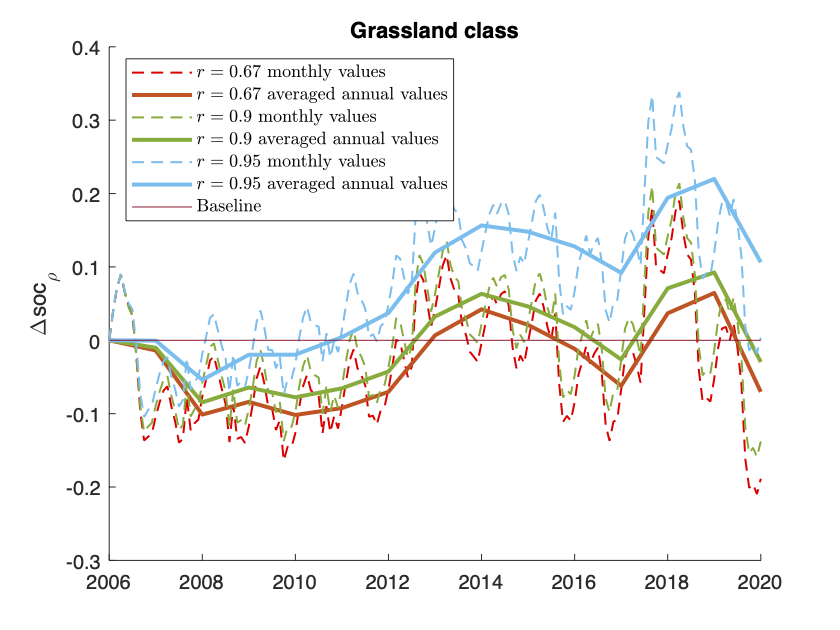}
  \end{center}
  \caption{The temporal evolution of $\Delta soc_{\rho^{(0)}(r)}(t^{(n)}_{m})$, together with its averaged annual values for grassland class. Parameters $r=0.67$, $r=0.9$, $r=0.95$. }\label{fig:grassland}
\end{figure}
\subsubsection{Arable class}
For the arable class, we firstly assume that no farmyard manure enter the system so that the evolution of $\Delta soc_{\rho^{(0)}(r)}(t^{(n)}_{m})$, together with its averaged annual values, is given in Figure \ref{fig:arable}. We set $r= 1$, $r=1.44$ (i.e. the value used in case of forest class in literature \cite{coleman1996rothc}), and $r=100$ in order to span all the values corresponding to this class.  This case  is the most critical one: the dynamics, even quantitatively different according to the values of $r\in r(a)$, is decreasing with this denoting a general trend departing from the baseline of positive values. For this class, in order to reach positive quantities, it is necessary to intensify the organic carbon input. 
To this aim we can apply the findings of  Theorem \ref{thm:fdit} in order to detect the optimal farmyard manure program  to enforce positive values of $\Delta soc_{\rho^{(0)}(r)}$. In  Figure \ref{fig:modifying} we report the temporal evolution of the modifying factor for farmyard manure $f_0(t)$, as defined in Definition \ref{defro}, for several values of $\epsilon$ spanning the interval $[0,\, 1]$. 
The effects of the fertilization process are shown in Figure \ref{fig:arable_controlled}.
\begin{figure}[h!] 
\begin{center} 
  \includegraphics[width=11cm]{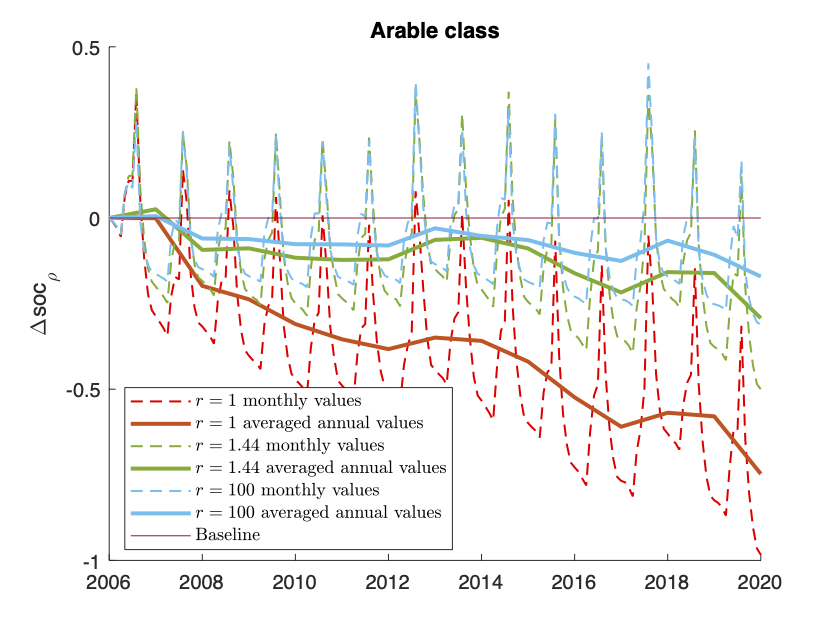}
  \end{center}
  \caption{The temporal evolution of $\Delta soc_{\rho^{(0)}(r)}(t^{(n)}_{m})$, together with its averaged annual values for the arable class. Parameters $r=1$, $r=1.44$, $r=100$. }\label{fig:arable}
\end{figure}
\begin{figure}[h!] 
\begin{center} 
  \includegraphics[width=11cm]{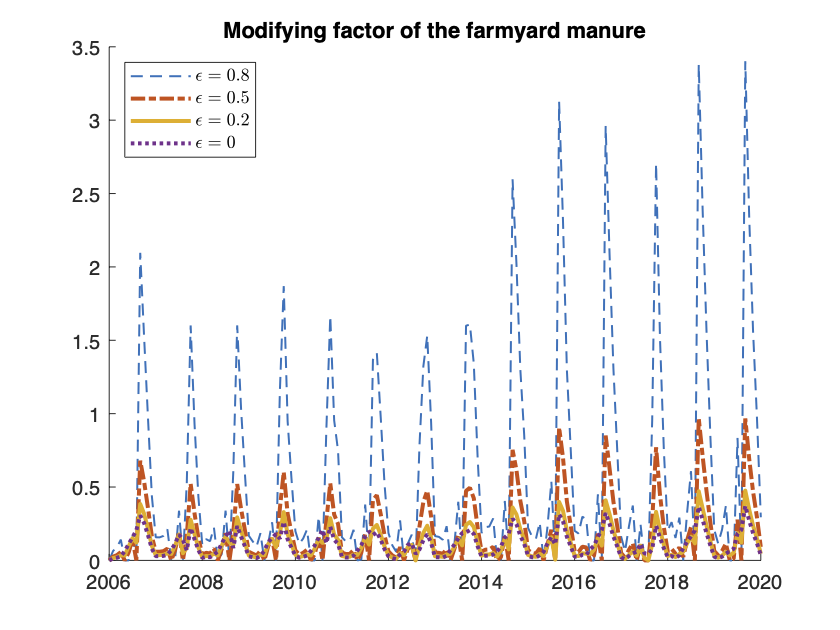}
  \end{center}
  \caption{The temporal evolution of modifying factor $f_0(t)$, for the arable class with $r=1$. Parameter $\epsilon=0.8,\, 0.5,\, 0.2,\, 0.$ }\label{fig:modifying}
\end{figure}

\begin{figure}[h!] 
\begin{center} 
  \includegraphics[width=11cm]{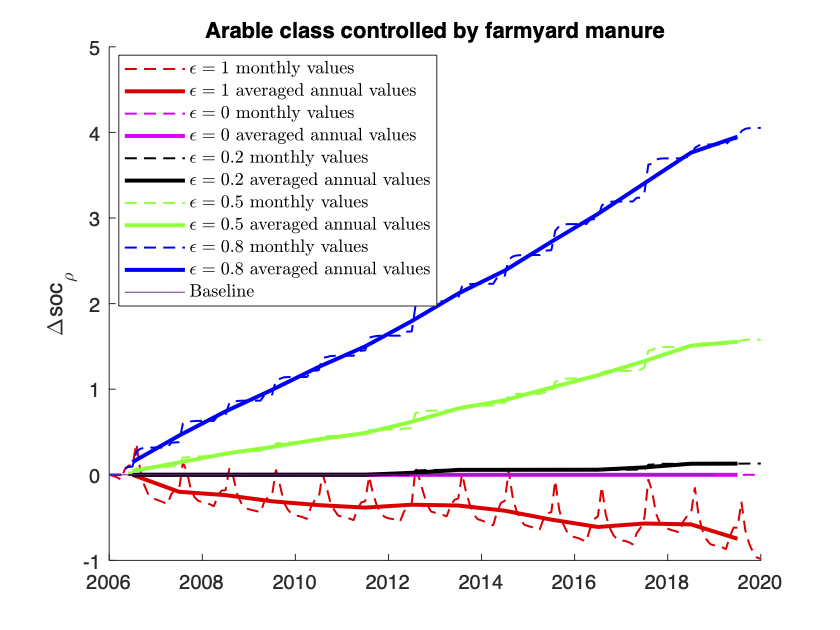}
  \end{center}
  \caption{The temporal evolution of $\Delta soc_{\rho^{(0)}(r)}(t^{(n)}_{m})$, together with its averaged annual values for the arable class with $r=1$ controlled by farmyard manure. Increasing values of $\Delta soc_{\rho^{(0)}(r)}(t^{(n)}_{m})$ for parameters $\epsilon= 0$, (no plant input), $\epsilon=  0.2,\, 0.5,\, 0.8$.  and $ \epsilon=1$ (no farmyard manure). }\label{fig:arable_controlled}
\end{figure}

\section{Comments and conclusion}\label{sec:8}
Soil carbon models (e.g. RothC \cite{coleman1996rothc}, Century \cite{parton1996century}) which take into account the interactions between climate and  land use management, are widely used to predict SOC changes under future climate scenarios.  Warmer temperatures  positively affect SOC stocks since they reduce decomposition, as an effect of a decreased soil moisture, and also increase Net Primary Production thus augmenting carbon inputs to the soil. On the other hand,  increasing temperatures negatively affect the SOC stocks as they increase the decomposition rate of soil organic matter. Hence, whether  soils gain or lose SOC, depends upon how balanced the competing gain and loss processes are, with subtle interacting changes in  temperature, moisture, soil type and land use \cite{gottschalk2012will}.

\noindent With the aim of  
improving the prediction of the factors that determine the size and direction of change, we have introduced  the so-called  {\it SOC change index} and we have described its evolution based on  the RothC carbon model. Under the hypothesis of constant environmental and organic fertilization conditions, it does not require to evaluate or measure the specific initial value of SOC,  as it describes the deviation from the assumed  initial equilibrium.

\noindent  The effectiveness of the novel index has been tested  for evaluating the  impact of warming temperatures on the achievement  of land degradation neutrality for the SOC indicator in Alta Murgia National Park, a protected area in the Apulia region located in the south of Italy.  The performed sensitivity analysis, based on time averaged parameter values, has provided local information on the impact of change in mean annual temperature,  of deviations of the mean annual NPP from its reference value and of the degree of decomposability of plant material.  The results of the sensitivity analysis is in  accordance with the experimental results, as we found that the SOC change index is negatively affected by increasing mean annual temperature and positively by increasing deviation of NPP. Changes in DPM/RPM ratio  $r$, which in turn are related to land use change, indicate that all land use classes are positively affected when deviation of NPP prevails on deviation in decomposition and negatively in the opposite case. In both cases the arable class results the most affected.

\noindent The simulated dynamics of the SOC change index in the Alta Murgia National Park in years $[2005,\, 2019]$ with climate data of CRU  (University  of  East Anglia) and  estimates of NPP taken  from  MOD17  project4,  indicate positive trends for forest and grassland classes. The arable class which is most affected by changes in NPP and temperature, as suggested by our sensitivity analysis, shows a negative trend. 
The dynamics of the SOC change index under the hypothesis of farmyard manure input has revealed a powerful tool for predicting the optimal land fertilization practice to implement for enhancing the  SOC stocks in the arable soil of Alta Murgia Park and invert the negative trend.

\noindent The construction of the SOC change index can be tailored on different soil carbon model dynamics. In particular, a future research direction is  represented  by the description of SOC change index under a suitable carbon model dynamics which  places the action of bacteria at the hearth of the  mechanisms of decomposition  process as indicated in \cite{lehmann2015contentious,hammoudi2015mathematical}.


%
%

\bibliographystyle{spmpsci}      
\bibliography{biblio}   

%
%

\clearpage
\section{Appendix}
\subsection{Thornthwaite's formula  for estimating the potential evapotranspiration }
 We need to estimate the 
      potential evapotranspiration  $pet(t)$, $[mm \, month^{-1}]$, estimated by means of the Thornthwaite's formula which is expressed, for the $n^{th}$ year, on a  monthly basis at the instants  $t_{m}^{(n)}\,:=\, t_0\,+nT\,+ \displaystyle \frac{T}{365} \displaystyle \sum_{i=1}^{m}\, N_i$ with $m = 1,\dots, 12$  and  $N_i$ denoting  the number of days of the $i^{th}$ month of the $n^{th}$ year\footnote{In a leap year $t_{m}^{(n)}\,:=\, t_0\,+nT\,+ \displaystyle \frac{T}{366} \displaystyle \sum_{i=1}^{m}\, N_i$  and $N_2=29$.}, as follows: 
     $$
     pet(t_{m}^{(n)}):= 16\, \displaystyle \frac{L_{d,m}^{(n)}}{12}\,\, \displaystyle \frac{N_m}{30} \,\left(\displaystyle \frac{10 \, Temp_{d,m}^{(n)} }{I_n}\right)^a. 
     $$
  In the above formula,  $L_{d,m}^{(n)}$ and $Temp_{d,m}^{(n)}$  represent  the average day length (hours) and the average daily temperature  of the  $m^{th}$ month of the $n^{th}$ year, respectively. Finally, $I_n$ is the heat index for the $n^{th}$ year given by 
     $$
     I_n\,=\, \displaystyle \sum_{k=1}^{12} \left(\displaystyle \frac{Temp_k^{(n)}}{5}\right)^{1.5}
     $$ 
     where $Temp_k^{(n)}:= \displaystyle \frac{\displaystyle \int_{t_{k-1}^{(n)}}^{t_{k}^{(n)}} Temp(s) \, ds}{t_{k}^{(n)}\,-\, t_{k-1}^{(n)}}$ is the $k^{th}$ monthly mean temperature, for $k=1,\dots, 12$. Finally,
     $$a\,=\, 6.7\, 10^{-7}\, I_n^3\,-\, 7.7 \, 10^{-5} \, I_n^2\, +\, 1.8\,  10^{-2} \, I_n \,+\, 0.49.$$
     
     \subsection{Estimation of the accumulate soil moisture deficit}
     The accumulate soil moisture deficit in the $n^{th}$ year, is also estimated on a  monthly basis at the instants  $t_{m}^{(n)}\,:=\, t_0\,+nT\,+ \displaystyle \frac{T}{365} \displaystyle \sum_{i=1}^{m}\, N_i$ with $m = 1,\dots, 12.$ 
     Then
      $Acc(t_{m}^{(n)},M)=0$ for all $m=1,\dots,\,\bar m$ 
such that $pet(t_{m}^{(n)})\,\leq  rain(t_{m}^{(n)})$, while
%
$$
     Acc(t_{m}^{(n)},M) =\,  \min\left(\max\left(M,\, \,Acc(t_{m-1}^{(n)},M) + rain(t_{m}^{(n)}) \,-\,pet(t_{m}^{(n)})\,\,\right),\,0\right)
     $$
for  $m=\bar m +1,\dots,T$. 

\end{document}